\documentclass{amsart}
\usepackage{amssymb,amsmath,amsthm}
\usepackage{mathrsfs}
\usepackage{tikz-cd}
\usepackage{enumitem}
\usepackage{mathtools}
\usepackage[titletoc,title]{appendix}

\usepackage{float}
\usepackage{pgf,tikz,pgfplots}
\pgfplotsset{compat=1.14}
\usepackage{graphicx}
\usepackage{caption}
\usetikzlibrary{arrows}

\theoremstyle{plain}
\newtheorem*{theorem*}{Theorem}
\newtheorem{theorem}{Theorem}[section]
\newtheorem{proposition}[theorem]{Proposition}
\newtheorem{lemma}[theorem]{Lemma}

\newtheorem{corollary}[theorem]{Corollary}
\newtheorem{conjecture}[theorem]{Conjecture}

\newtheorem{combprob}[theorem]{Combinatorial problem}
\newtheorem*{principle*}{Guiding principle}
\newtheorem{question}[theorem]{Question}
\newtheorem*{conjecture*}{Conjecture}

\theoremstyle{definition}
\newtheorem{definition}[theorem]{Definition}

\newtheorem{example}[theorem]{Example}

\DeclareMathOperator{\codim}{codim}

\newcommand{\ZZ}{\mathbb{Z}}
\newcommand{\AS}{\mathbb{A}}

\newcommand{\CC}{\mathbb{C}}
\newcommand{\PP}{\mathbb{P}}
\newcommand{\GG}{\mathbb{G}}
\newcommand{\TT}{\mathbb{T}}

\begin{document}

\title{Secant indices of projective varieties}
\author{Grayson Jorgenson}
\date{}
\maketitle

\begin{abstract}
To each subvariety $X$ in $\PP^n$ of codimension $m$ we associate an integer sequence of length $m + 1$ from $1$ to the degree of $X$ recording the maximal cardinalities of finite,  reduced intersections of $X$ with linear subvarieties of $\PP^n$. We call this the sequence of secant indices of $X$. Similar numbers have been studied independently with the aim of classifying subvarieties with extremal secant spaces. Our focus in this note is the study of the combinatorial properties that the secant indices satisfy collectively. We show these sequences are strictly increasing for nondegenerate smooth subvarieties, develop a method to compute term-wise lower bounds for the secant indices, and compute these lower bounds for Veronese and Segre varieties. In the case of Veronese varieties, the truth of the Eisenbud-Green-Harris conjecture would imply the lower bounds we find are in fact equal to the secant indices. Along the way we state several relevant questions and additional conjectures which to our knowledge are open.
\end{abstract}

\section{Introduction}

Given a pure-dimensional closed subscheme $X$ of $\PP^n$ of codimension $m$, and a choice of integer $0\leq j \leq m$, let $\Lambda_j(X)\subseteq \GG(j, n)$ denote the subset of dimension $j$ linear subvarieties of $\PP^n$ with finite and reduced intersection with $X$. Then one may define $$\mathfrak{L}_j(X) := \max\{|X\cap L|\mid L\in \Lambda_j(X)\}.$$ 

Here $|\cdot|$ is used to denote the set-theoretic count of the closed points of the scheme within. We call this integer $\mathfrak{L}_j(X)$ the $j$th \emph{secant index} of $X$, and together these numbers form a sequence of length $m+1$ starting at $1$ and ending at $\deg(X)$, which we denote by $\mathfrak{L}(X)$.

Similar numbers have been studied independently, such as in \cite{ahn_kwak1} \cite{bertin1}  \cite{kwak1} \cite{nollet1} \cite{noma1}. It is known that the presence of a $m$-multisecant line, a line meeting $X$ finitely in at least $m$ points, implies that the Castelnuovo-Mumford regularity $\text{reg}(X)$ of $X$ is at least $m$ \cite[Proposition 1.1]{kwak1}. The cited references focus on classifying varieties with extremal secant subspaces with one of the goals being to provide examples of varieties with near maximal regularity. Such work provides evidence for the Eisenbud-Goto regularity conjecture \cite{eisenbud_goto1} that $\text{reg}(X)\leq \deg(X) - m + 1$ when $X$ is a nondegenerate codimension $m$ subvariety of $\PP^n$.

In this work, we are instead concerned with the properties that these indices satisfy collectively, as sequences. The sequences that occur in general seem to admit interesting combinatorial descriptions. For instance, given a smooth nondegenerate variety $X$, $\mathfrak{L}(X)$ is always strictly increasing. Yet if $X$ is not a subvariety of minimal degree in $\PP^n$, then this sequence must contain gaps and these gaps need not occur only in one place in the sequence. One of the simplest examples of this is when $X$ is the image of the degree $3$ Veronese embedding of $\PP^2$ into $\PP^9$; the sequence $\mathfrak{L}(v_3^{(2)}(\PP^2))$ is this case has $8$ terms, but there are $9$ numbers in the sequence $1,2,\ldots,9$. Here we have $\mathfrak{L}(v_3^{(2)}(\PP^2)) = (1,2,3,4,5,6,7,9)$ and the observation that the gap occurs between the final two terms of the sequence is tantamount to the classical Cayley-Bacharach theorem.

We attempt the development of a method to compute the secant indices for an arbitrary smooth variety that involves defining two accessory sequences to $\mathfrak{L}(X)$, denoted by ${\mathfrak{R}\mathfrak{L}}^G(X)$ and $\mathfrak{R}\mathfrak{L}(X)$, that are term-wise lower bounds for $\mathfrak{L}(X)$. In the specific cases where $X$ is a Veronese or Segre variety, we show that computing these accessory sequences is equivalent to solving two purely combinatorial problems. These accessory sequences then produce conjectural values for $\mathfrak{L}(X)$ for those varieties. In the case of Veronese varieties, we will show that the truth of the Eisenbud-Green-Harris conjecture \cite{eisenbud_green_harris1} would imply the sequences ${\mathfrak{R}\mathfrak{L}}^G(X)$, $\mathfrak{R}\mathfrak{L}(X)$ are indeed equal to $\mathfrak{L}(X)$. The definitions of ${\mathfrak{R}\mathfrak{L}}^G(X)$ and $\mathfrak{R}\mathfrak{L}(X)$ work for arbitrary smooth varieties, and the hope is that for every variety $X$ there is a tractable combinatorial problem associated to computing these sequences, reflecting the nature of intersections of hyperplane sections of $X$ which are as reducible as possible.

This note is organized as follows. In Section 2 we describe properties that the sequence of secant indices satisfy in general, including the strictness of their growth in the case of a smooth nondegenerate subvariety of $\PP^n$. We state a number of questions about the secant indices in Section 3 which to our knowledge are open, and define the integer sequences ${\mathfrak{R}\mathfrak{L}}^G(X)$, $\mathfrak{R}\mathfrak{L}(X)$. In Section 4 we illustrate the accessory sequences and show they are equal when $X$ is a Veronese variety. We prove a method that computes them and show that it would follow from the Eisenbud-Green-Harris conjecture that the sequences agree with $\mathfrak{L}(X)$. In Section 5, we derive a method to compute ${\mathfrak{R}\mathfrak{L}}^G(X)$ when $X$ is a Segre variety and provide several example computations, conjecturing there that ${\mathfrak{R}\mathfrak{L}}^G(X)$ agrees with $\mathfrak{L}(X)$. In Section 6 we discuss our original motivation from the problem of counting lines on surfaces, where our hope is that one can compute these maximal numbers by extrapolation from a related sequence of indices.

\textbf{Acknowledgments} I wish to thank Paolo Aluffi for his support and for many useful discussions.


\section{General properties}

Throughout we work over $k = \CC$, all points are assumed to be closed points, and a variety is an integral scheme. The most basic form of the main question we study in this note is as follows: what is the maximum number of points at which a linear subvariety of a given dimension can meet a pure-dimensional reduced subscheme $X\subseteq \PP^n$? This is only interesting when the linear subvariety has dimension small enough to meet $X$ in finitely many points, thus we ask about the numbers $$\max \{\deg(X\cap L)\mid L\subseteq \PP^n \text{ linear, } \dim(L) = i, \dim(X\cap L)= 0\},$$ for $0\leq i \leq \codim(X).$ Because a general linear subvariety of dimension less than $\codim(X)$ fails to meet $X$ at all, the problem of determining these numbers is a type of \emph{quasi-enumerative problem} \cite[Sectioin 11.2]{eisenbud_harris1}. Note the first of these numbers, the maximum number of points that a single point can meet $X$, is clearly just $1$.

However, for the other numbers the exact method of counting the points in the intersections must be made precise. It would be nice to know that the last number in this sequence is always $\deg(X)$. If we count the points of a zero-dimensional scheme $Y$ by letting $\deg(Y)$ denote the scheme-theoretic degree rather than the set-theoretic count of the distinct closed points of the support of $Y$, this is not always the case. 

\begin{example}
Consider the union $X$ of two planes in $\PP^4$ meeting at a single point. The degree of $X$ is $2$, but it is well-known that $\deg(X\cap L) = 3$ for any plane $L$ meeting $X$ at only its singular point. 
\end{example}

To avoid such issues we opt to instead use the naive set-theoretic count, considering the numbers $$\max \{|X\cap L|\mid L\subseteq \PP^n \text{ linear, } \dim(L) = j, \dim(X\cap L)= 0\},$$ for $0\leq j \leq \codim(X).$ We still reserve $\deg(\cdot)$ to denote scheme-theoretic degree, and instead denote by $|Y|$ the cardinality of the set of closed points of a zero-dimensional scheme $Y$. The resulting sequence will always be nondecreasing. This version of the question fell out from our original motivation which is discussed in Section 6.

One of the properties that seems reasonable to expect is that for nondegenerate $X$ this sequence should in fact be strictly increasing. To prove such a result we introduce one last refinement: we require the intersections we are counting to be reduced, arriving at the sequence of \emph{secant indices} of $X$, $\mathfrak{L}(X)$, defined in Section 1, and we will restrict our attention to smooth projective varieties. There is then no distinction between using the set-theoretic count or the scheme-theoretic degree. In this section and Sections 3, 4, and 5, the sequences sporting these additional properties are our objects of study.

\begin{example}
\label{ratnormcurve}
Consider the rational normal curve $C$ that is the image of the Veronese map $v^{(1)}_d: \PP^1\hookrightarrow \PP^d$, a degree $d$ smooth curve. The sequence $\mathfrak{L}(C)$ has $d$ terms, starting at $1$ and ending at $d$. It is indeed the only strictly increasing sequence of that length connecting those two numbers, $\mathfrak{L}(C) = (1,2,\ldots, d)$. Veronese embeddings of higher dimension projective spaces will be discussed in Section 4. 
\end{example} 

\begin{example}
\label{degenvars}
If $Y$ is a smooth subvariety of $\PP^n$, and $Y$ is contained in a linear subvariety of dimension $r$, then the last $n - r$ terms of $\mathfrak{L}(Y)$ are all equal to $\deg(Y)$.
\end{example}

\subsection{Strictness of growth}

The sequence of secant indices of a smooth variety $X$ in $\PP^n$ is always nondecreasing. If $X$ is nondegenerate, then the sequence is in fact strictly increasing. If $X$ is degenerate, then $\mathfrak{L}(X)$ has repeated terms as in Example \ref{degenvars}, but is otherwise strictly increasing. We derive this property below using an elementary argument revolving around Bertini's theorem. Note this argument is similar in essence to the elementary approach seen in \cite{kwak1} for deriving the upper bound on the secant indices, and indeed that same upper bound is an immediate consequence of the strictness of growth we derive here. Several of the lemmas used below are well-known results, but for lack of appropriate references we give complete proofs for most of them here.

The specific consequence of Bertini's theorem \cite[Theorem 0.5]{eisenbud_harris1} we will invoke throughout is the following.

\begin{lemma}
\label{specificbertini}
Let $L$ be a linear subvariety and let $X$ be a smooth subvariety of $\PP^n$. The set of all hyperplanes in $\PP^n$ containing $L$ forms a linear subvariety $T$ of dual projective space $(\PP^n)^\vee$. There is a nonempty open subset of $T$ of hyperplanes containing $L$ and having smooth intersection with $X$ outside of $X\cap L$.
\end{lemma} 

For any closed subscheme $Y$ of $\PP^n(x_0:\ldots :x_n)$ and any point $p\in Y$, we denote by $\mathbb{T}_p(Y)$ the \emph{embedded tangent space} to $Y$ at $p$. If $Y = V(I)$ for a homogeneous ideal $I$ with generators $F_1,\ldots,F_r\in k[x_0,\ldots, x_n]$, then the embedded tangent space at $p$ is the linear subvariety cut out by the polynomials $$\frac{\partial F_i}{\partial x_0}(p) x_0 + \ldots + \frac{\partial F_i}{\partial x_n}(p) x_n,$$ for $i = 1,\ldots, r$. The subscheme $Y$ is smooth at $p$ if and only if $$\dim(\mathbb{T}_p(Y)) = \dim(Y)$$ by the Jacobian criterion for singularities. Following immediately from this definition: 

\begin{lemma}
\label{tangspacecomm}
Let $Y$ be a closed subscheme of $\PP^n$, and let $L$ be a linear subvariety. Suppose $p$ is a closed point of the scheme-theoretic intersection $Y\cap L$. Then $\mathbb{T}_p(Y\cap L) = \mathbb{T}_p(Y)\cap L$.
\end{lemma} 

\begin{lemma}
\label{tangentcut}
Let $Y$ be a closed subscheme of $\PP^n$ and let $L$ be a linear subvariety of dimension $< n - 1$. Suppose $L$ has reduced zero-dimensional intersection with $Y$, and the intersection consists of the points $p_1,\ldots,p_r$. Then if $T$ is the linear subvariety of $(\PP^n)^\vee$ consisting of all hyperplanes containing $L$, there is an nonempty open subset of $T$ of hyperplanes $H$ with $Y\cap H$ smooth at the points $p_1,\ldots,p_r$.
\end{lemma}



Verifying that $\mathfrak{L}(X)$ is strictly increasing is not difficult when $X$ is a curve. First, note that a reduced hyperplane section of a nondegenerate variety is nondegenerate, inside of the hyperplane. We say a reducible subscheme of projective space is nondegenerate if not all of its irreducible components lie in one hyperplane.

\begin{lemma}
\label{gennondegen}
Let $X$ be a nondegenerate subvariety of $\PP^n$ of dimension $> 0$ and suppose $H$ is a hyperplane such that $X\cap H$ is reduced. Then $X\cap H$ is nondegenerate in $H\cong \PP^{n-1}$.
\end{lemma}
\begin{proof} 
Suppose to the contrary that there is a hyperplane $L$ in $H\cong \PP^{n-1}$ with $X\cap H\subseteq L$. Since $X$ is nondegenerate, there exists a point $p\in X\setminus H$. Thus we can find a hyperplane $H^\prime$ of $\PP^n$ containing both $L$ and $p$. Further, $H\cap H^\prime = L$, so $X\cap H\subseteq X\cap H^\prime$.

The intersection $X\cap H^\prime$ is pure-dimensional of dimension $\dim(X) - 1$, thus $p$ is a point on an irreducible component $Y$ of $X\cap H^\prime$ of that dimension. Then $H\cap Y$ must be a proper closed subset of $Y$, thus of smaller dimension, and therefore contained in one of the irreducible components of $X\cap H$.

Let $X_1,\ldots,X_r, Y, Z_1,\ldots, Z_m$ be the irreducible components of $X\cap H^\prime$, all considered with reduced scheme structures. Here $X_1,\ldots, X_r$ are the irreducible components of $X\cap H$.

Since $X\cap H$ is reduced, we have $$\deg(X\cap H) = \sum_{i = 1}^r\deg(X_i),$$ while $$\deg(X\cap H^\prime) = \sum_{i = 1}^r m_{X_i}(X, H^\prime) \deg(X_i) + m_{Y}(X, H^\prime)\deg(Y) + \sum_{i = 1}^m m_{Z_i}(X,H^\prime)\deg(Z_i).$$ The notation $m_Z(A,B)$ stands for the intersection multiplicity of the intersection of two varieties $A,B$ along an irreducible component $Z$ of $A\cap B$, as in \cite{eisenbud_harris1}.

This is a contradiction since the intersection multiplicities are positive and we must have $\deg(X\cap H) = \deg(X\cap H^\prime)$, see \cite[Theorem I.7.7]{hartshorne1}.
\end{proof}

\begin{lemma}
\label{curvestrict}
Let $C$ be a nondegenerate, smooth, and irreducible curve in $\PP^n$. Then $\mathfrak{L}(X)$ is strictly increasing.
\end{lemma}
\begin{proof}
We must show that given any linear subvariety $L$ of dimension $\dim(L) < n - 1$ with reduced intersection with $X$, there is a linear subvariety $L^\prime$ of dimension one greater that also has reduced intersection with $C$ such that $|C\cap L| < |C\cap L^\prime|$. 

The set of all hyperplanes containing $L$ is a linear subvariety of $(\PP^n)^\vee$ which induces a positive-dimensional linear system on $C$. By Bertini's theorem and Lemma~\ref{tangentcut}, the general hyperplane in this linear system has reduced intersection with $C$.

Pick any such hyperplane $H$. By Lemma \ref{gennondegen}, the points of $C\cap H$ span $H$. Thus $L$ cannot contain all of the points in $C\cap H$; there must be at least one point $p$ in $(C\cap H)\setminus (C\cap L)$.

If $\dim(L) = n - 2$, then we are done; if $L$ is taken to be a linear subvariety realizing $\mathfrak{L}_{n - 2}(C)$, then we have shown $\mathfrak{L}_{n - 2}(C) < \mathfrak{L}_{n - 1}(C)$.

Otherwise, we can pick a linear subvariety $L^\prime$ of dimension one greater than $\dim(L)$ and with $L\subseteq L^\prime \subseteq H$, and $p\in L^\prime$. Thus $\mathfrak{L}_{\dim(L)}(C) < \mathfrak{L}_{\dim(L) + 1}(C)$.
\end{proof}

The idea to get the general result is to reduce to the case of a curve when dealing with a higher dimensional variety.

\begin{theorem}
\label{genstrict}
Suppose $X$ is a smooth nondegenerate subvariety of $\PP^n$. Then $\mathfrak{L}(X)$ is strictly increasing.
\end{theorem}
\begin{proof}
By Lemma \ref{curvestrict}, we may assume $\dim(X) > 1$. Suppose that $L$ is a linear subvariety of dimension $r < \codim(X)$ so that $X\cap L$ is reduced and zero-dimensional, and $|X\cap L| = \mathfrak{L}_r(X)$.

By Lemma \ref{specificbertini} and Lemma \ref{tangentcut}, there is a hyperplane $H$ containing $L$ so that $X\cap H$ is smooth. The Fulton-Hansen connectedness theorem \cite{fulton_hansen1} implies any hyperplane section of $X$ is connected, so the hypothesis that $X\cap H$ is smooth implies it is also irreducible. Lemma~\ref{gennondegen} then shows $X\cap H$ is nondegenerate as a subvariety of $H$.

Thus by induction, we may assume the existence of a linear subvariety $T$ containing $L$ of dimension $n - \dim(X) + 1$ so that $T\cap X$ is a smooth, irreducible, nondegenerate curve in $T$. By Lemma \ref{curvestrict}, there exists a linear subvariety $L\subseteq L^\prime\subseteq T$ of dimension $r + 1$ such that $$|X\cap T\cap L^\prime| > |X\cap T\cap L|.$$ Therefore $$\mathfrak{L}_{r+1}(X) > \mathfrak{L}_r(X).$$
\end{proof}

Note, as a minor consequence, this gives a slightly different way to think about the degree lower bound that all nondegenerate projective subvarieties satisfy, see for instance \cite{eisenbud1}. For smooth nondegenerate $X$, the fact that $\mathfrak{L}(X)$ is strictly increasing forces $\deg(X) \geq \codim(X) + 1$.

That $X$ is a variety is also essential. Secant indices of smooth, pure-dimensional, nondegenerate, and reduced subschemes do not necessarily form strictly increasing sequences.

\begin{example}
Consider the smooth curve $C$ in $\PP^3$ that is the union of three skew lines $L_1,L_2,L_3$ all passing through another line $L$. Then $\mathfrak{L}_1(C) = 3$, as $L\cap C$ consists of three distinct points, but any plane that contains $L$ and meets $C$ at points outside of $L\cap C$ must contain one of $L_1,L_2,L_3$. So $\mathfrak{L}(C) = (1,3,3)$ in this case.
\end{example}

One other immediate and basic consequence of the strictly increasing property is recovering the known upper bound for the cardinality of intersections with extremal secant spaces.

\begin{proposition}
\label{seqbound}
Tautologically, $$\mathfrak{L}_i(X) = \deg(X) - \sum_{j = i}^{\codim(X) - 1}(\mathfrak{L}_{j + 1}(X) - \mathfrak{L}_{j}(X)),$$ for $i = 0,\ldots, \codim(X) - 1$. Thus in particular, if $X$ is a smooth nondegenerate subvariety of $\PP^n$, each difference $\mathfrak{L}_{j + 1}(X) - \mathfrak{L}_{j}(X)$ is at least $1$, and so $$\mathfrak{L}_i(X) \leq \deg(X) - \codim(X) + i$$ for each $i$.
\end{proposition}

See also Kwak \cite{kwak1}. When $i = 1$, this bound has been used as evidence for the Eisenbud-Goto regularity conjecture.

\section{Questions and a guiding principle}

One of our main interests is finding a means of computing the sequence of secant indices for a given smooth subvariety $X\subseteq \PP^n$ but this seems difficult in general. It is clear that one method of obtaining a term-wise lower bound for $\mathfrak{L}(X)$ is to take a linear subvariety $L\subseteq \PP^n$ of dimension $\codim(X)$ so that $X\cap L$ is finite and reduced, and compute the sequence $\mathfrak{L}(X\cap L)$ where $X\cap L$ is considered as a subscheme of $L\cong \PP^{\codim(X)}$.

First, for two integer sequences $(a_j)_{j = 1}^r, (b_j)_{j = 1}^r$ of the same length $r$, we write $(a_j)_{j = 1}^r \preceq (b_j)_{j = 1}^r$ if $a_j \leq b_j$ for each $j$. This is a partial order on the set of all integer sequences of the same length. Additionally, we can define a total order on that set by stating $(a_j)_{j = 1}^r \leq (b_j)_{j = 1}^r$ if and only if either the sequences are equal, or there is a $1\leq k \leq r$ such that $a_k < b_k$, and $a_j = b_j$ for every $j > k$. 

\begin{proposition}
Let $X\subseteq\PP^n$ be any smooth subvariety, and let $H\subset \PP^n$ be any hyperplane not containing $X$ so that $X\cap H$ is also smooth. Then $$\mathfrak{L}(X\cap H)\preceq \mathfrak{L}(X)$$ where $X\cap H$ is considered as a subvariety of $H\cong \PP^{n-1}$.
\end{proposition} 

A natural question then is to ask when is this all that needs to be done.

\begin{question}
\label{realizationquestion}
For what subvarieties $X$ of $\PP^n$ is $\mathfrak{L}(X)$ realized by $\mathfrak{L}(X\cap L)$ for a linear subvariety $L$ of dimension $\codim(X)$ such that $X\cap L$ is finite and reduced? Further, when is $\mathfrak{L}(X)$ realized as a sequence of the form $$(|X\cap (H_1\cap \ldots\cap H_n)|,\ldots,|X\cap (H_1\cap\ldots\cap H_{\dim(X)})|),$$ for linearly independent hyperplanes $H_1,\ldots, H_n$?
\end{question}

\begin{question}
\label{greedyquestion}
Consider $Y := X\cap L$ for a linear subvariety $L$ of dimension $\codim(X)$ with $X\cap L$ reduced and finite. We can define two integer sequences:
\begin{enumerate}
\item $$(\max\{|Y\cap L_0|\},\ldots,\max\{|Y\cap L_{\codim(X)}|\}),$$
\item $$\max\{(|Y\cap L_0|,\ldots,|Y\cap L_{\codim(X)}|)\}.$$
\end{enumerate} Here each $L_j$ denotes a linear subvariety of dimension $j$ contained inside $L$. The maximums in (1) are taken over all possible $L_j$, and the maximum in (2) is taken over all chains $L_0\subseteq \ldots\subseteq L_{\codim(X)}$ using the total order $\leq$ defined above. Is it always the case that these two sequences are the same?
 \end{question}

Note that the answer to Question \ref{greedyquestion} is negative for arbitrary finite subsets $Y$ of $\PP^n$.

\begin{example}
Let $Y$ be a finite set consisting of $3$ points $p_1,p_2,p_3$ on a line $T$ and $5$ points $q_1,\ldots,q_5$ on a plane $H$ in $L = \PP^3$. Suppose $T$ is not contained in $H$, and $p_1,p_2,p_3$ are not the point of intersection $T\cap H$. Suppose also that the five points on $H$ are arranged so that no three of them are collinear, and that no two of them lie on a line containing $T\cap H$. Then the sequence from (1) of Question \ref{greedyquestion} is $$(1, 3, 5, 8),$$ and the sequence from (2) is $$(1, 2, 5, 8).$$
\end{example}

At least for the case of Veronese varieties $X$ considered in Section 4, the truth of the Eisenbud-Green-Harris conjecture would positively answer Question \ref{realizationquestion}. In the same section we will also prove that the two sequences of Question \ref{greedyquestion} are equal when $X$ is a Veronese variety, as a consequence of the Clements-Lindstr\"om theorem.

For any finite set $Y$ of points in $\PP^n$, computing $\mathfrak{L}(Y)$ is equivalent to considering the dimensions of the spans of all subsets of of $Y$, considered as subsets of points in the vector space $k^{n+1}$. Another question then becomes to ask for a subvariety $X$ of $\PP^n$ about what possible linear dependences between the $\deg(X)$ points of a reduced and finite intersection $X\cap L$ where $L$ is linear of dimension $\codim(X)$ occur as $L$ is varied among all such linear subvarieties. One could phrase this in terms of matroids.

\begin{question}
Let $X\subseteq \PP^n$ be a variety. For each linear subvariety $L$ of dimension $\codim(X)$ and $X\cap L$ finite and reduced, the subsets of linearly independent points of $X\cap L$ considered inside $k^{n+1}$ form a matroid. What matroids can be realized in this way? 
\end{question}

Finally, what sequences of integers can be realized as a sequence of secant indices?

\begin{question}
If $X$ is a smooth, nondegenerate subvariety of $\PP^n$, $\mathfrak{L}(X)$ is a strictly increasing sequence of integers from $1$ to $\deg(X)$ of length $\codim(X) + 1$. What strictly increasing sequences of this length from $1$ to $\deg(X)$ occur in this way?
\end{question}

\begin{question}
For $X$ for which $\deg(X)$ exceeds $\codim(X) + 1$, in what positions and in what sizes do the gaps in $\mathfrak{L}(X)$ occur?
\end{question}

To our knowledge, the above questions have received little prior study, if any at all. For this last question, some of the related work on classifying varieties with extremal secant spaces provides a partial result about the existence of a gap between the last two terms of $\mathfrak{L}(X)$. One such result is the following due to Kwak \cite[Proposition 3.2]{kwak1}.

\begin{proposition}
Let $X$ be a nondegenerate subvariety of $\PP^n$ of dimension $\geq 1$, and codimension $\geq 2$. If $X$ has an extremal curvilinear secant subspace in at least one of the dimensions $1,\ldots, \codim(X) - 1$, then $X$ is 
\begin{enumerate}
\item a Veronese surface in $\PP^5$,
\item a projected Veronese surface in $\PP^4$,
\item a rational scroll, that is, a projective bundle over a smooth curve. 
\end{enumerate}
\end{proposition}

As an immediate corollary of this result, we can classify all varieties that do not have a gap between the final two terms of $\mathfrak{L}(\cdot)$.

\begin{corollary}
\label{classificationextremal}
Let $X$ be a nondegenerate smooth subvariety of $\PP^n$ of codimension $\geq 2$. Then unless $X$ is a rational scroll, the Veronese surface in $\PP^5$, or a projected Veronese surface in $\PP^4$, $$\mathfrak{L}_{\codim(X)}(X) - \mathfrak{L}_{\codim(X) - 1}(X) \geq 2.$$ 
\end{corollary}
\begin{proof}
The only work that needs to be done is reconcile our language with that used in the cited reference. A linear subvariety $L$ is said to be a \emph{curvilinear secant subspace to $X$} if $X\cap L$ if finite with each point of $X\cap L$ locally contained in a smooth curve on $X$. This last criterion is equivalent to specifying that $\dim(\TT_p(X)\cap L)\leq 1$ for each point $p\in X\cap L$. Such an $L$ is called \emph{extremal} if its intersection contains the maximal possible number of points, counted with appropriate multiplicity, $\text{length}(X\cap L) = \deg(X) - \codim(X) + \dim(L)$.

Our point of view in this note predominantly takes the more naive route of considering only reduced intersections; we only consider linear subvarieties $L$ with $X\cap L$ reduced and finite. This ensures that $\dim(L\cap T_p(X)) = 0$ for each $p\in X\cap L$ and thus, in particular, such an $L$ is a curvilinear secant subspace to $X$.

Therefore, the result of Kwak implies that the inequality seen in Proposition \ref{seqbound} is in fact strict for each $i = 0,\ldots, \codim(X) - 1$.
\end{proof}

Kwak's result along with similar work \cite{noma1} on bounding the maximal possible lengths can be used in this way to treat the question of whether there is a gap between the final two terms of $\mathfrak{L}(X)$ for all smooth nondegenerate varieties. But it does not seem possible to use these results to provide lower bounds for the terms of the sequence, or to say more about the size of the penultimate gap and the presence of other gaps in the sequence.

Returning to the goal of computing $\mathfrak{L}(X)$, a natural attempt to reduce the complexity of this computation is to compute the indices that arise when we only use a subset of the possible linear subvarieties. In particular, it seems reasonable to expect that the secant indices are the same if we were to only consider linear subvarieties cut out by hyperplanes that meet $X$ in the most ``reducible way'' possible. This leads us to define two additional accessory sequences to $\mathfrak{L}(X)$ which in some cases, such as those considered in Sections 4 and 5, become tractable to compute. To get a precise notion that generalizes beyond Veronese and Segre varieties we will make several definitions.

\begin{definition}
Let $X$ be a smooth nondegenerate subvariety of $\PP^n$ of dimension $r$.
\begin{itemize}
\item We say $X$ is \emph{$p$-reducible} if there exists a collection of hyperplanes $H_1,\ldots, H_n$ so that their common intersection is a single point, $H_1\cap \ldots \cap H_r \cap X$ is finite and reduced, each $H_j\cap X$ for $j = 1,\ldots, r$ is reduced and has exactly $p$ distinct irreducible components, and finally $|H_1\cap\ldots\cap H_j\cap X| > |H_1\cap\ldots\cap H_{j+1}\cap X|$ for each $j = r,\ldots, n-1$. Any such sequence of hyperplanes is said to \emph{satisfy the conditions of $p$-reducibility}.
\item We call the maximal $p$ such that $X$ is $p$-reducible the \emph{reducibility} of $X$.
\item Suppose $X$ has reducibility $p$. Denote by $\Lambda_j^R(X) \subseteq \Lambda_j(X)$ the subset of linear subvarieties of dimension $j$ cut out by hyperplanes $H_1,\ldots, H_{n - j}$ such that the $H_1,\ldots,H_{n-j}$ are the initial part of a sequence of hyperplanes satisfying the conditions of $p$-reducibility.
\end{itemize}
\end{definition}

Note that by Bertini's theorem, every smooth nondegenerate subvariety of dimension $> 1$ is at least $1$-reducible.

\begin{proposition}
Let $X\subseteq \PP^n$ be a smooth nondegenerate subvariety of dimension $> 1$. Then $X$ is $1$-reducible.
\end{proposition}
\begin{proof}
First, we can find hyperplanes $H_1,\ldots,H_{\dim(X)}$ so that each $X\cap H_j$ is smooth and irreducible and so that $H_1\cap\ldots\cap H_{\dim(X)}\cap X$ is finite and reduced by the classical Bertini theorem.

Let $L = H_1\cap\ldots\cap H_{\dim(X)}$ and note that by Lemma \ref{gennondegen} the points of $X\cap L$ span $L$ since $X\cap L$ is reduced. Choose a subset of the points of $X\cap L$ that span a linear subvariety $L^\prime$ of dimension $\dim(L) - 1$. Then by Lemmas \ref{specificbertini}, \ref{tangentcut}, we may find a hyperplane $H_{\dim(X) + 1}$ not containing $L$ but containing $L^\prime$ so that $X\cap H_{\dim(X) + 1}$ is smooth and irreducible. Since $L^\prime$ has smaller dimension than $L$, we see that $$|H_1\cap\ldots\cap H_{\dim(X)}\cap X| > |H_1\cap\ldots\cap H_{\dim(X) + 1}\cap X|.$$

Repeating this process completes the needed sequence of hyperplanes satisfying the conditions of $1$-reducibility.
\end{proof}

In general, computing the reducibility of a variety seems to be an independently interesting question. However, in the cases studied in Sections 4, 5, there is no mystery about the reducibility of the varieties in consideration. The reducibility of the Segre variety $\sigma(\PP^n\times\PP^m)$ is $2$ when $n, m > 0$, and that of the Veronese variety $v^{(n)}_{d}(\PP^n)$ is $d$.

At this point, one could define a new sequence which is a lower bound for $\mathfrak{L}(X)$ at each term by modifying the definition of $\mathfrak{L}_i(X)$ to only use linear subvarieties from $\Lambda^R(X)$ instead of from $\Lambda(X)$. This sequence seems interesting, but still appears challenging to compute. One difficulty in computing this new sequence is controlling how many irreducible components must be considered. To define the two accessory sequences to $\mathfrak{L}(X)$ we will restrict the number of components.

\begin{definition}
Let $X\subseteq\PP^n$ be a smooth nondegenerate subvariety with reducibility $p$. Consider the set of all sequences of hyperplanes $H_1,\ldots,H_n$ satisfying the conditions of $p$-reducibility.
\begin{itemize}
\item For each, consider the number of irreducible components in the union $\bigcup (H_j\cap X)$ which have nonempty intersection with $H_1\cap\ldots\cap H_{\dim(X)}\cap X$. Denote by $\mu(X)$ the minimal number of such irreducible components attained by the union of one of these sequences.
\item Denote by $\mathfrak{H}(X)$ the set of all sequences of hyperplanes $(H_1,\ldots,H_n)$ satisfying the conditions of $p$-reducibility and with the number of irreducible components of $\bigcup (H_j\cap X)$, which each have nonempty intersection with $H_1\cap\ldots\cap H_{\dim(X)}\cap X$, equal to $\mu(X)$.
\end{itemize}
\end{definition}

Note that $\mu(X) \geq p\dim(X)$ always, as the number of irreducible components in $\bigcup_{i = 1}^{\dim(X)}H_i\cap X$ is $p\dim(X)$ for any sequence of hyperplanes $(H_1,\ldots,H_n)$ satisfying the conditions of $p$-reducibility. Once more, it seems to be an interesting question for arbitrary $X$ what the value of $\mu(X)$ is. However, the for both the Veronese and Segre varieties we will consider, the minimal possible value $p\dim(X)$ is attained. We now are able to define the two accessory sequences to $\mathfrak{L}(X)$.

\begin{definition}
Let $X$ be a smooth nondegenerate subvariety of $\PP^n$ with reducibility $p$. Then if $\codim(X) = m$, we define for each $0\leq j\leq m$ an integer $${\mathfrak{R}\mathfrak{L}}_j(X) := \max \{|X\cap L|\mid L = \bigcap_{i = 1}^{n - j} H_i, (H_1,\ldots,H_n)\in \mathfrak{H}(X)\},$$ called the $j$th \emph{reducible secant index}. As for the original secant indices, we denote these numbers collectively by $$\mathfrak{R}\mathfrak{L}(X) := ({\mathfrak{R}\mathfrak{L}}_0(X),\ldots,{\mathfrak{R}\mathfrak{L}}_m(X)).$$
\end{definition}

\begin{definition}
We define the sequence of \emph{greedy reducible secant indices} of $X$ as $${\mathfrak{R}\mathfrak{L}}^G(X) := \max \{(|X\cap (H_1\cap \ldots \cap H_n)|,\ldots, |X\cap (H_1\cap \ldots\cap H_{\dim(X)})|)\},$$ where this maximum is taken over all $(H_1,\ldots,H_n)\in \mathfrak{H}(X)$ and the sequences of integers are compared using the total order $\leq$ defined previously.
\end{definition}

Altogether, we see that $${\mathfrak{R}\mathfrak{L}}^G(X) \preceq {\mathfrak{R}\mathfrak{L}}(X) \preceq \mathfrak{L}(X).$$

It seems conceivable that for many subvarieties $X$ these three sequences are in fact equal. In Section 4 we will compute ${\mathfrak{R}\mathfrak{L}}^G(X), {\mathfrak{R}\mathfrak{L}}(X)$ when $X$ is a Veronese variety, and will show they are equal. There we show also that the truth of the Eisenbud-Green-Harris conjecture would imply all three sequences are equal for Veronese varieties. With so little evidence however, we refer to the expectation in general that these sequences should all be equal as a ``guiding principle'' rather than a conjecture.

\begin{principle*}
For any smooth nondegenerate subvariety $X\subseteq\PP^n$, $${\mathfrak{R}\mathfrak{L}}^G(X) = {\mathfrak{R}\mathfrak{L}}(X) = \mathfrak{L}(X).$$
\end{principle*}

\section{Veronese varieties}

In this section we consider the secant indices of the images of the Veronese embeddings. Throughout we will denote by $v^{(n)}_d : \PP^n\hookrightarrow \PP^{N_d^{(n)} - 1}$ the degree $d$ Veronese embedding of $\PP^n$ into $\PP^{N_d^{(n)} - 1}$, where $N_d^{(n)} = \binom{n + d}{n}$. In Example \ref{ratnormcurve}, we saw that $\mathfrak{L}(v^{(1)}_d(\PP^1)) = (1,2,\ldots, d)$ for every $d > 0$.

Similarly, by Theorem \ref{genstrict}, since the Veronese surface $v_2^{(2)}(\PP^2)$ has degree $4$ in $\PP^5$, that is, it is a minimal degree subvariety, we see that $\mathfrak{L}(v_2^{(2)}(\PP^2)) = (1,2,3,4)$. The degree of the image of the Veronese embedding $v_d^{(n)}$ is $d^n$, and these sequences begin to become interesting when the degree exceeds the codimension of the image by more than one.

The first case where this happens is for $X = v_3^{(2)}(\PP^2)$. This is a degree $9$ subvariety of codimension $7$ in $\PP^9$. Thus $\mathfrak{L}(X)$ has $8$ terms, but there are $9$ numbers in the sequence $1,2,3,\ldots,9$. Therefore, the sequence of secant indices of $X$ must contain exactly one gap, of size $2$. The only remaining question is where in the sequence does this gap occur. In this case, the position of the gap is explained by the classical Cayley-Bacharach theorem \cite{eisenbud_green_harris2}.

\begin{theorem}
Let $C_1, C_2$ be two cubic curves in $\PP^2$ not sharing any irreducible component and meeting in $9$ distinct points. Suppose $C_3$ is another cubic curve containing $8$ of those $9$ points. Then $C_3$ contains all $9$ points.
\end{theorem}

Together with Theorem \ref{genstrict}, this proves that $\mathfrak{L}(X) = (1,2,3,4,5,6,7,9)$. From this point of view, the content of the Cayley-Bacharach theorem is that the gap occurs between the final two terms of $\mathfrak{L}(X)$.

The correspondence in use here is that the hyperplane sections of $v_d^{(n)}(\PP^n)$ are exactly the degree $d$ hypersurfaces in $\PP^n.$ The maximum number of irreducible components such a hypersurface can have is $d$, and it is clear that the reducibility of $v_d^{(n)}(\PP^n)$ is always $d$.

In what follows we give a method to compute ${\mathfrak{R}\mathfrak{L}}^G(v_d^{(n)}(\PP^n))$ for all $n, d > 1$. Our motivation in defining the sequences of greedy reducible secant indices becomes clear in this context since the problem of computing ${\mathfrak{R}\mathfrak{L}}^G(v_d^{(n)}(\PP^n))$ can be converted into a tractable combinatorial question. We illustrate the computation first with the example $X = v_3^{(2)}(\PP^2)$.

\begin{example}
Consider two fully reducible hyperplane sections of $X\subseteq \PP^9$ whose intersection consists of exactly $9$ distinct points. These hyperplane sections are curves in $\PP^2(x:y:z)$, say $C_1 = V((x - z)(x - 2z)(x - 3z))$ and $C_2 = V((y - z)(y - 2z)(y - 3z))$.

We will compute ${\mathfrak{R}\mathfrak{L}}^G(X)$. The top term is $9$, so then to find the next term we need to pick a hyperplane in $\PP^9$ independent to the first two. The ``greediness'' of the sequence ${\mathfrak{R}\mathfrak{L}}^G(X)$ is from the fact that we form it by finding hyperplane sections which remove the fewest points from the remaining finitely many points in the intersection at each step.

The number of irreducible components in the union $C_1\cup C_2$ is $6 = \mu(X)$. To compute the next term of ${\mathfrak{R}\mathfrak{L}}^G(X)$ we must find an independent hyperplane section $C_3$ which removes the fewest amount of points from $C_1\cap C_2$, but does not violate the $\mu(X)$ condition. That is, the way to interpret the defining constraint of the hyperplane sections used in the construction of ${\mathfrak{R}\mathfrak{L}}^G(X)$ is that they are those which do not introduce any additional irreducible components that meet the points we are working with; in other words, we must form $C_3$ out of some choice of at most three distinct irreducible components from $C_1, C_2$. 

These three components cannot all be from the same $C_j$, so it is clear that the best we can do is take, for example, $C_3 = V((x - z)(x - 2z)(y - z))$. Then $C_1\cap C_2\cap C_3$ is reduced and consists of $7$ points. Note also that so long as the number of points in the intersection decreases, the latest hyperplane section cannot be a linear combination of the previous. This process is continued to compute the remaining numbers. One possible continuation is illustrated in Figure 1.

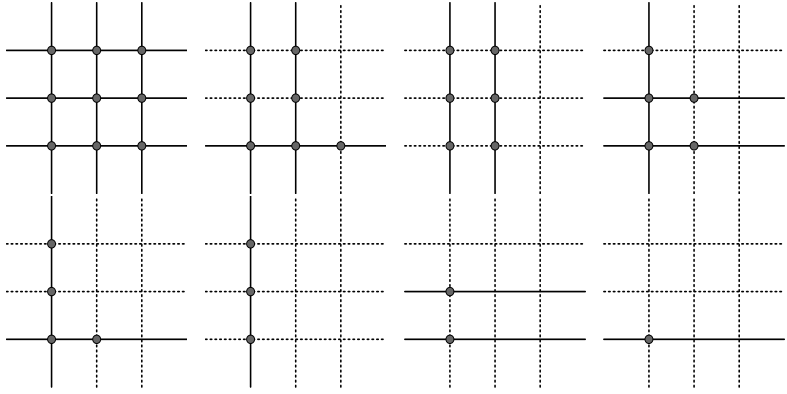
\begin{figure}
\centering

\begin{minipage}[b]{0.2\textwidth}
\resizebox{1in}{1in}{%
\definecolor{wwwwww}{rgb}{0.4,0.4,0.4}
\begin{tikzpicture}[line cap=round,line join=round,>=triangle 45,x=1cm,y=1cm]
\clip(0,0) rectangle (4,4);
\draw [line width=1pt,solid,domain=0:4] plot(\x,{(--2-0*\x)/2});
\draw [line width=1pt,solid,domain=0:4] plot(\x,{(--4-0*\x)/2});
\draw [line width=1pt,solid,domain=0:4] plot(\x,{(--6-0*\x)/2});
\draw [line width=1pt,solid] (3,0) -- (3,4);
\draw [line width=1pt,solid] (2,0) -- (2,4);
\draw [line width=1pt,solid] (1,0) -- (1,4);
\begin{scriptsize}
\draw [fill=wwwwww] (1,1) circle (2.5pt);
\draw [fill=wwwwww] (3,1) circle (2.5pt);
\draw [fill=wwwwww] (1,2) circle (2.5pt);
\draw [fill=wwwwww] (3,2) circle (2.5pt);
\draw [fill=wwwwww] (1,3) circle (2.5pt);
\draw [fill=wwwwww] (3,3) circle (2.5pt);
\draw [fill=wwwwww] (2,3) circle (2.5pt);
\draw [fill=wwwwww] (2,1) circle (2.5pt);
\draw [fill=wwwwww] (2,2) circle (2.5pt);
\end{scriptsize}
\end{tikzpicture}
}
\end{minipage}
\begin{minipage}[b]{0.2\textwidth}
\resizebox{1in}{1in}{%
\definecolor{wwwwww}{rgb}{0.4,0.4,0.4}
\begin{tikzpicture}[line cap=round,line join=round,>=triangle 45,x=1cm,y=1cm]
\clip(0,0) rectangle (4,4);
\draw [line width=1pt,solid,domain=0:4] plot(\x,{(--2-0*\x)/2});
\draw [line width=1pt,dotted,domain=0:4] plot(\x,{(--4-0*\x)/2});
\draw [line width=1pt,dotted,domain=0:4] plot(\x,{(--6-0*\x)/2});
\draw [line width=1pt,dotted] (3,0) -- (3,4);
\draw [line width=1pt,solid] (2,0) -- (2,4);
\draw [line width=1pt,solid] (1,0) -- (1,4);
\begin{scriptsize}
\draw [fill=wwwwww] (1,1) circle (2.5pt);
\draw [fill=wwwwww] (3,1) circle (2.5pt);
\draw [fill=wwwwww] (1,2) circle (2.5pt);
\draw [fill=wwwwww] (1,3) circle (2.5pt);
\draw [fill=wwwwww] (2,3) circle (2.5pt);
\draw [fill=wwwwww] (2,1) circle (2.5pt);
\draw [fill=wwwwww] (2,2) circle (2.5pt);
\end{scriptsize}
\end{tikzpicture}
}
\end{minipage}
\begin{minipage}[b]{0.2\textwidth}
\resizebox{1in}{1in}{%
\definecolor{wwwwww}{rgb}{0.4,0.4,0.4}
\begin{tikzpicture}[line cap=round,line join=round,>=triangle 45,x=1cm,y=1cm]
\clip(0,0) rectangle (4,4);
\draw [line width=1pt,dotted,domain=0:4] plot(\x,{(--2-0*\x)/2});
\draw [line width=1pt,dotted,domain=0:4] plot(\x,{(--4-0*\x)/2});
\draw [line width=1pt,dotted,domain=0:4] plot(\x,{(--6-0*\x)/2});
\draw [line width=1pt,dotted] (3,0) -- (3,4);
\draw [line width=1pt,solid] (2,0) -- (2,4);
\draw [line width=1pt,solid] (1,0) -- (1,4);
\begin{scriptsize}
\draw [fill=wwwwww] (1,1) circle (2.5pt);
\draw [fill=wwwwww] (1,2) circle (2.5pt);
\draw [fill=wwwwww] (1,3) circle (2.5pt);
\draw [fill=wwwwww] (2,3) circle (2.5pt);
\draw [fill=wwwwww] (2,1) circle (2.5pt);
\draw [fill=wwwwww] (2,2) circle (2.5pt);
\end{scriptsize}
\end{tikzpicture}
}
\end{minipage}
\begin{minipage}[b]{0.2\textwidth}
\resizebox{1in}{1in}{%
\definecolor{wwwwww}{rgb}{0.4,0.4,0.4}
\begin{tikzpicture}[line cap=round,line join=round,>=triangle 45,x=1cm,y=1cm]
\clip(0,0) rectangle (4,4);
\draw [line width=1pt,solid,domain=0:4] plot(\x,{(--2-0*\x)/2});
\draw [line width=1pt,solid,domain=0:4] plot(\x,{(--4-0*\x)/2});
\draw [line width=1pt,dotted,domain=0:4] plot(\x,{(--6-0*\x)/2});
\draw [line width=1pt,dotted] (3,0) -- (3,4);
\draw [line width=1pt,dotted] (2,0) -- (2,4);
\draw [line width=1pt,solid] (1,0) -- (1,4);
\begin{scriptsize}
\draw [fill=wwwwww] (1,1) circle (2.5pt);
\draw [fill=wwwwww] (1,2) circle (2.5pt);
\draw [fill=wwwwww] (1,3) circle (2.5pt);
\draw [fill=wwwwww] (2,1) circle (2.5pt);
\draw [fill=wwwwww] (2,2) circle (2.5pt);
\end{scriptsize}
\end{tikzpicture}
}
\end{minipage}

\begin{minipage}[b]{0.2\textwidth}
\resizebox{1in}{1in}{%
\definecolor{wwwwww}{rgb}{0.4,0.4,0.4}
\begin{tikzpicture}[line cap=round,line join=round,>=triangle 45,x=1cm,y=1cm]
\clip(0,0) rectangle (4,4);
\draw [line width=1pt,solid,domain=0:4] plot(\x,{(--2-0*\x)/2});
\draw [line width=1pt,dotted,domain=0:4] plot(\x,{(--4-0*\x)/2});
\draw [line width=1pt,dotted,domain=0:4] plot(\x,{(--6-0*\x)/2});
\draw [line width=1pt,dotted] (3,0) -- (3,4);
\draw [line width=1pt,dotted] (2,0) -- (2,4);
\draw [line width=1pt,solid] (1,0) -- (1,4);
\begin{scriptsize}
\draw [fill=wwwwww] (1,1) circle (2.5pt);
\draw [fill=wwwwww] (1,2) circle (2.5pt);
\draw [fill=wwwwww] (1,3) circle (2.5pt);
\draw [fill=wwwwww] (2,1) circle (2.5pt);
\end{scriptsize}
\end{tikzpicture}
}
\end{minipage}
\begin{minipage}[b]{0.2\textwidth}
\resizebox{1in}{1in}{%
\definecolor{wwwwww}{rgb}{0.4,0.4,0.4}
\begin{tikzpicture}[line cap=round,line join=round,>=triangle 45,x=1cm,y=1cm]
\clip(0,0) rectangle (4,4);
\draw [line width=1pt,dotted,domain=0:4] plot(\x,{(--2-0*\x)/2});
\draw [line width=1pt,dotted,domain=0:4] plot(\x,{(--4-0*\x)/2});
\draw [line width=1pt,dotted,domain=0:4] plot(\x,{(--6-0*\x)/2});
\draw [line width=1pt,dotted] (3,0) -- (3,4);
\draw [line width=1pt,dotted] (2,0) -- (2,4);
\draw [line width=1pt,solid] (1,0) -- (1,4);
\begin{scriptsize}
\draw [fill=wwwwww] (1,1) circle (2.5pt);
\draw [fill=wwwwww] (1,2) circle (2.5pt);
\draw [fill=wwwwww] (1,3) circle (2.5pt);
\end{scriptsize}
\end{tikzpicture}
}
\end{minipage}
\begin{minipage}[b]{0.2\textwidth}
\resizebox{1in}{1in}{%
\definecolor{wwwwww}{rgb}{0.4,0.4,0.4}
\begin{tikzpicture}[line cap=round,line join=round,>=triangle 45,x=1cm,y=1cm]
\clip(0,0) rectangle (4,4);
\draw [line width=1pt,solid,domain=0:4] plot(\x,{(--2-0*\x)/2});
\draw [line width=1pt,solid,domain=0:4] plot(\x,{(--4-0*\x)/2});
\draw [line width=1pt,dotted,domain=0:4] plot(\x,{(--6-0*\x)/2});
\draw [line width=1pt,dotted] (3,0) -- (3,4);
\draw [line width=1pt,dotted] (2,0) -- (2,4);
\draw [line width=1pt,dotted] (1,0) -- (1,4);
\begin{scriptsize}
\draw [fill=wwwwww] (1,1) circle (2.5pt);
\draw [fill=wwwwww] (1,2) circle (2.5pt);
\end{scriptsize}
\end{tikzpicture}
}
\end{minipage}
\begin{minipage}[b]{0.2\textwidth}
\resizebox{1in}{1in}{%
\definecolor{wwwwww}{rgb}{0.4,0.4,0.4}
\begin{tikzpicture}[line cap=round,line join=round,>=triangle 45,x=1cm,y=1cm]
\clip(0,0) rectangle (4,4);
\draw [line width=1pt,solid,domain=0:4] plot(\x,{(--2-0*\x)/2});
\draw [line width=1pt,dotted,domain=0:4] plot(\x,{(--4-0*\x)/2});
\draw [line width=1pt,dotted,domain=0:4] plot(\x,{(--6-0*\x)/2});
\draw [line width=1pt,dotted] (3,0) -- (3,4);
\draw [line width=1pt,dotted] (2,0) -- (2,4);
\draw [line width=1pt,dotted] (1,0) -- (1,4);
\begin{scriptsize}
\draw [fill=wwwwww] (1,1) circle (2.5pt);
\end{scriptsize}
\end{tikzpicture}
}
\end{minipage}
\caption{successive hyperplane sections of $v_3^{(2)}(\PP^2)$}
\end{figure}

In Figure 1 the images are ordered left to right, top to bottom. The first image is of the $9$ points in the intersection of $C_1$ and $C_2$. The curve $C_1$ consists of the union of the three vertical lines, and the curve $C_2$ of the three horizontal lines. These curves are all visualized by working within the affine chart $\PP^2\setminus V(z)$ and over the real numbers. The next image depicts $C_3$ and the $7$ points remaining after intersecting it with $C_1\cap C_2$. From there, the hyperplane sections chosen are $V((x - z)(x - 2z)z), V((x - z)(y - z)(y - 2z)), V((x - z)(y - z)z), V((x - z)z^2), V((y - z)(y - 2z)z), V((y - z)z^2),$ in that order. Note the presence of the irreducible component $V(z)$ in some of these hyperplane sections, used to ensure the curve is of the correct degree. As $V(z)$ does not meet $C_1\cap C_2$, these additional hyperplane sections satisfy the conditions of $d$-reducibility and the $\mu(X) = 6$ condition.
\end{example}

This example was too simple to illustrate some questions that need to be resolved when computing ${\mathfrak{R}\mathfrak{L}}^G(\cdot)$ in more complicated situations. One is, at a given step of the computation, if there are multiple possible new hyperplane sections to use, which each take away the same number of points, does it matter which one is used? Additionally, how can one determine whether the greedy sequence matches the sequence ${\mathfrak{R}\mathfrak{L}}(\cdot)$?

For Veronese varieties, answering these questions becomes easier if we first move the computations to a simplified combinatorial setting. The heuristic guiding the transformation is that the number of points we are getting after adding hyperplane sections is the same if we were to ``collapse'' the distinct irreducible components into one irreducible component, counted with a multiplicity. We will rigorously justify this heuristic later with Lemma \ref{reductionofmonideal}.

So in the above example, for the purposes of the computation, we would collapse $C_1$ into $V(x^3)$, $C_2$ into $V(y^3)$, and $C_3$ into $V(x^2y)$. The scheme-theoretic degree of their intersection is $$\deg(C_1\cap C_2\cap C_3) = \deg(V(x^3, y^2, x^2y)) = 7.$$

Therefore, a combinatorial problem we could consider is the following.

\begin{combprob}
\label{combprob1}
Determine the maximal dimension of $k[x_0,\ldots,x_{n-1}]/I$ where $I$ is the dehomogenization with respect to $x_n$ of a monomial ideal $(x_0^d,\ldots, x_{n-1}^d) + J$ in $k[x_0,\ldots,x_{n}]$ generated by $r$ linearly independent degree $d$ monomials, for each $r\geq n$.
\end{combprob}

\begin{example}
\label{quarticexample}
We address this new problem for $n = 2$, $d = 4$. When $r = 2$, we obtain the vector space $k[x_0,x_1]/(x_0^4, x_1^4)$, of dimension $16$. We see that the sequence of maximal dimensions is $(1, 2, 3, 4, 5, 6, 7, 8, 9, 10, 12, 13, 16)$, organized from $r = 12$ to $r = 2$. Figure 2 shows how this sequence can be obtained by sequentially adding monomials to the ideal $J$ using a staircase diagram.

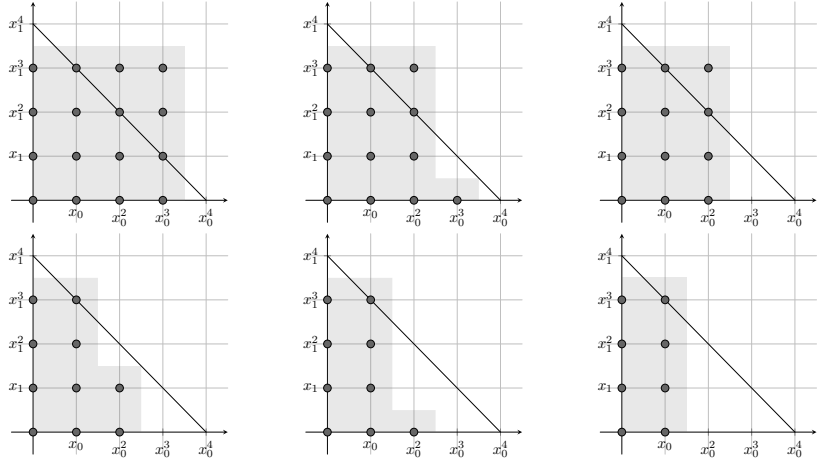
\begin{figure}
\definecolor{wwwwww}{rgb}{0.4,0.4,0.4}

\begin{minipage}[b]{0.3\textwidth}
\resizebox{1.2in}{1.2in}{%
\begin{tikzpicture}[line cap=round,line join=round,>=triangle 45,x=1cm,y=1cm]
\begin{axis}[
x=1cm,y=1cm,
axis lines=middle,
ymajorgrids=true,
xmajorgrids=true,
xmin=-0.5,
xmax=4.5,
ymin=-0.5,
ymax=4.5,
xtick={0,...,4},
ytick={0,...,4},
xticklabels={$1$, $x_0$, $x_0^2$, $x_0^3$, $x_0^4$},
yticklabels={$1$, $x_1$, $x_1^2$, $x_1^3$, $x_1^4$}]
\clip(-0.5,-0.5) rectangle (4.5,4.5);

\fill[line width=0.4pt,color=wwwwww,fill=wwwwww,fill opacity=0.15] (0,3.5) -- (3.5,3.5) -- (3.5,0) -- (0,0) -- cycle;

\draw [line width=0.5pt] (0,4)-- (4,0);
\begin{scriptsize}
\draw [fill=wwwwww] (3,3) circle (2.5pt);
\draw [fill=wwwwww] (3,2) circle (2.5pt);
\draw [fill=wwwwww] (3,1) circle (2.5pt);
\draw [fill=wwwwww] (3,0) circle (2.5pt);
\draw [fill=wwwwww] (2,1) circle (2.5pt);
\draw [fill=wwwwww] (2,2) circle (2.5pt);
\draw [fill=wwwwww] (2,3) circle (2.5pt);
\draw [fill=wwwwww] (1,3) circle (2.5pt);
\draw [fill=wwwwww] (1,2) circle (2.5pt);
\draw [fill=wwwwww] (1,1) circle (2.5pt);
\draw [fill=wwwwww] (0,0) circle (2.5pt);
\draw [fill=wwwwww] (0,1) circle (2.5pt);
\draw [fill=wwwwww] (0,2) circle (2.5pt);
\draw [fill=wwwwww] (0,3) circle (2.5pt);
\draw [fill=wwwwww] (1,0) circle (2.5pt);
\draw [fill=wwwwww] (2,0) circle (2.5pt);
\end{scriptsize}
\end{axis}
\end{tikzpicture}
}
\end{minipage}
\begin{minipage}[b]{0.3\textwidth}
\resizebox{1.2in}{1.2in}{%
\begin{tikzpicture}[line cap=round,line join=round,>=triangle 45,x=1cm,y=1cm]
\begin{axis}[
x=1cm,y=1cm,
axis lines=middle,
ymajorgrids=true,
xmajorgrids=true,
xmin=-0.5,
xmax=4.5,
ymin=-0.5,
ymax=4.5,
xtick={0,...,4},
ytick={0,...,4},
xticklabels={$1$, $x_0$, $x_0^2$, $x_0^3$, $x_0^4$},
yticklabels={$1$, $x_1$, $x_1^2$, $x_1^3$, $x_1^4$}]
\clip(-0.5,-0.5) rectangle (4.5,4.5);

\fill[line width=0.4pt,color=wwwwww,fill=wwwwww,fill opacity=0.15] (0,3.5) -- (2.5,3.5) -- (2.5,0.5) -- (3.5, 0.5) -- (3.5, 0) -- (0,0) -- cycle;

\draw [line width=0.5pt] (0,4)-- (4,0);
\begin{scriptsize}
\draw [fill=wwwwww] (3,0) circle (2.5pt);
\draw [fill=wwwwww] (2,1) circle (2.5pt);
\draw [fill=wwwwww] (2,2) circle (2.5pt);
\draw [fill=wwwwww] (2,3) circle (2.5pt);
\draw [fill=wwwwww] (1,3) circle (2.5pt);
\draw [fill=wwwwww] (1,2) circle (2.5pt);
\draw [fill=wwwwww] (1,1) circle (2.5pt);
\draw [fill=wwwwww] (0,0) circle (2.5pt);
\draw [fill=wwwwww] (0,1) circle (2.5pt);
\draw [fill=wwwwww] (0,2) circle (2.5pt);
\draw [fill=wwwwww] (0,3) circle (2.5pt);
\draw [fill=wwwwww] (1,0) circle (2.5pt);
\draw [fill=wwwwww] (2,0) circle (2.5pt);
\end{scriptsize}
\end{axis}
\end{tikzpicture}
}
\end{minipage}
\begin{minipage}[b]{0.3\textwidth}
\resizebox{1.2in}{1.2in}{%
\begin{tikzpicture}[line cap=round,line join=round,>=triangle 45,x=1cm,y=1cm]
\begin{axis}[
x=1cm,y=1cm,
axis lines=middle,
ymajorgrids=true,
xmajorgrids=true,
xmin=-0.5,
xmax=4.5,
ymin=-0.5,
ymax=4.5,
xtick={0,...,4},
ytick={0,...,4},
xticklabels={$1$, $x_0$, $x_0^2$, $x_0^3$, $x_0^4$},
yticklabels={$1$, $x_1$, $x_1^2$, $x_1^3$, $x_1^4$}]
\clip(-0.5,-0.5) rectangle (4.5,4.5);

\fill[line width=0.4pt,color=wwwwww,fill=wwwwww,fill opacity=0.15] (0,3.5) -- (2.5,3.5) -- (2.5,0) -- (0,0) -- cycle;

\draw [line width=0.5pt] (0,4)-- (4,0);
\begin{scriptsize}
\draw [fill=wwwwww] (2,1) circle (2.5pt);
\draw [fill=wwwwww] (2,2) circle (2.5pt);
\draw [fill=wwwwww] (2,3) circle (2.5pt);
\draw [fill=wwwwww] (1,3) circle (2.5pt);
\draw [fill=wwwwww] (1,2) circle (2.5pt);
\draw [fill=wwwwww] (1,1) circle (2.5pt);
\draw [fill=wwwwww] (0,0) circle (2.5pt);
\draw [fill=wwwwww] (0,1) circle (2.5pt);
\draw [fill=wwwwww] (0,2) circle (2.5pt);
\draw [fill=wwwwww] (0,3) circle (2.5pt);
\draw [fill=wwwwww] (1,0) circle (2.5pt);
\draw [fill=wwwwww] (2,0) circle (2.5pt);
\end{scriptsize}
\end{axis}
\end{tikzpicture}
}
\end{minipage}

\begin{minipage}[b]{0.3\textwidth}
\resizebox{1.2in}{1.2in}{%
\begin{tikzpicture}[line cap=round,line join=round,>=triangle 45,x=1cm,y=1cm]
\begin{axis}[
x=1cm,y=1cm,
axis lines=middle,
ymajorgrids=true,
xmajorgrids=true,
xmin=-0.5,
xmax=4.5,
ymin=-0.5,
ymax=4.5,
xtick={0,...,4},
ytick={0,...,4},
xticklabels={$1$, $x_0$, $x_0^2$, $x_0^3$, $x_0^4$},
yticklabels={$1$, $x_1$, $x_1^2$, $x_1^3$, $x_1^4$}]
\clip(-0.5,-0.5) rectangle (4.5,4.5);

\fill[line width=0.4pt,color=wwwwww,fill=wwwwww,fill opacity=0.15] (0, 3.5) -- (1.5, 3.5) -- (1.5, 1.5) -- (2.5, 1.5) -- (2.5, 0) -- (0,0) -- cycle;

\draw [line width=0.5pt] (0,4)-- (4,0);
\begin{scriptsize}
\draw [fill=wwwwww] (2,1) circle (2.5pt);
\draw [fill=wwwwww] (1,3) circle (2.5pt);
\draw [fill=wwwwww] (1,2) circle (2.5pt);
\draw [fill=wwwwww] (1,1) circle (2.5pt);
\draw [fill=wwwwww] (0,0) circle (2.5pt);
\draw [fill=wwwwww] (0,1) circle (2.5pt);
\draw [fill=wwwwww] (0,2) circle (2.5pt);
\draw [fill=wwwwww] (0,3) circle (2.5pt);
\draw [fill=wwwwww] (1,0) circle (2.5pt);
\draw [fill=wwwwww] (2,0) circle (2.5pt);
\end{scriptsize}
\end{axis}
\end{tikzpicture}
}
\end{minipage}
\begin{minipage}[b]{0.3\textwidth}
\resizebox{1.2in}{1.2in}{%
\begin{tikzpicture}[line cap=round,line join=round,>=triangle 45,x=1cm,y=1cm]
\begin{axis}[
x=1cm,y=1cm,
axis lines=middle,
ymajorgrids=true,
xmajorgrids=true,
xmin=-0.5,
xmax=4.5,
ymin=-0.5,
ymax=4.5,
xtick={0,...,4},
ytick={0,...,4},
xticklabels={$1$, $x_0$, $x_0^2$, $x_0^3$, $x_0^4$},
yticklabels={$1$, $x_1$, $x_1^2$, $x_1^3$, $x_1^4$}]
\clip(-0.5,-0.5) rectangle (4.5,4.5);

\fill[line width=0.4pt,color=wwwwww,fill=wwwwww,fill opacity=0.15] (0, 3.5) -- (1.5, 3.5) -- (1.5, 0.5) -- (2.5, 0.5) -- (2.5, 0) -- (0,0) -- cycle;

\draw [line width=0.5pt] (0,4)-- (4,0);
\begin{scriptsize}
\draw [fill=wwwwww] (1,3) circle (2.5pt);
\draw [fill=wwwwww] (1,2) circle (2.5pt);
\draw [fill=wwwwww] (1,1) circle (2.5pt);
\draw [fill=wwwwww] (0,0) circle (2.5pt);
\draw [fill=wwwwww] (0,1) circle (2.5pt);
\draw [fill=wwwwww] (0,2) circle (2.5pt);
\draw [fill=wwwwww] (0,3) circle (2.5pt);
\draw [fill=wwwwww] (1,0) circle (2.5pt);
\draw [fill=wwwwww] (2,0) circle (2.5pt);
\end{scriptsize}
\end{axis}
\end{tikzpicture}
}
\end{minipage}
\begin{minipage}[b]{0.3\textwidth}
\resizebox{1.2in}{1.2in}{%
\begin{tikzpicture}[line cap=round,line join=round,>=triangle 45,x=1cm,y=1cm]
\begin{axis}[
x=1cm,y=1cm,
axis lines=middle,
ymajorgrids=true,
xmajorgrids=true,
xmin=-0.5,
xmax=4.5,
ymin=-0.5,
ymax=4.5,
xtick={0,...,4},
ytick={0,...,4},
xticklabels={$1$, $x_0$, $x_0^2$, $x_0^3$, $x_0^4$},
yticklabels={$1$, $x_1$, $x_1^2$, $x_1^3$, $x_1^4$}]
\clip(-0.5,-0.5) rectangle (4.5,4.5);

\fill[line width=0.4pt,color=wwwwww,fill=wwwwww,fill opacity=0.15] (0, 3.5) -- (1.5, 3.5) -- (1.5, 0) -- (0,0) -- cycle;

\draw [line width=0.5pt] (0,4)-- (4,0);
\begin{scriptsize}
\draw [fill=wwwwww] (1,3) circle (2.5pt);
\draw [fill=wwwwww] (1,2) circle (2.5pt);
\draw [fill=wwwwww] (1,1) circle (2.5pt);
\draw [fill=wwwwww] (0,0) circle (2.5pt);
\draw [fill=wwwwww] (0,1) circle (2.5pt);
\draw [fill=wwwwww] (0,2) circle (2.5pt);
\draw [fill=wwwwww] (0,3) circle (2.5pt);
\draw [fill=wwwwww] (1,0) circle (2.5pt);
\end{scriptsize}
\end{axis}
\end{tikzpicture}
}
\end{minipage}
\caption{successive monomial additions to $J$ for $n = 2$, $d = 4$}
\end{figure}

This figure is read from left to right, top to bottom like Figure 1. The first image is that of the $16$ monomials representing the basis elements of $k[x_0,x_1]/(x_0^4,x_1^4)$. These monomials correspond to points in the diagram in accordance to their exponent vectors. Adding a monomial to $J$ amounts to removing that monomial and all others it divides from the diagram; the number of monomials leftover is the dimension of the new polynomial ring quotient. We may only add monomials to $J$ that are on or below the pictured diagonal line, as these are exactly those of degree $\leq 4$. This setup leads to a ``game'' wherein one tries to remove as few monomials as possible from the diagram with each successive addition to $J$. The pictured sequence is part of the sequence of additions $x_0^3 x_1, x_0^3, x_0^2x_1^2, x_0^2x_1, x_0^2, x_0x_1^3, x_0x_1^2, x_0x_1, x_0, x_1^3, x_1^2, x_1$. The additions past $x_0^2$ are not pictured as there are no more gaps larger than $1$ that appear in the sequences of dimensions. Note that this sequence of monomials is part of the lexicographic sequence of degree $4$ monomials in $x_0,x_1,x_2$, but dehomogenized with respect to $x_2$. A priori, this lexicographic sequence only gives a ``greedy'' sequence of removals, removing as few monomials from the diagram at each step as possible. This example is simple enough to check by brute force that the greedy sequence is in fact the term-wise maximal sequence. Showing that this is true for the more complicated examples when $n$ or $d$ is larger requires a more efficient argument.
\end{example}

We will show that this combinatorial problem is equivalent to computing both ${\mathfrak{R}\mathfrak{L}}^G(v_d^{(n)}(\PP^n))$, $\mathfrak{R}\mathfrak{L}(v_d^{(n)}(\PP^n))$. In particular, in this example, we computed ${\mathfrak{R}\mathfrak{L}}^G(v_4^{(2)}(\PP^2))$.

As a brief aside, note the following generalization of the Cayley-Bacharach theorem \cite{eisenbud_green_harris1}.

\begin{theorem}
Let $\Gamma\subseteq \PP^r$ be a complete intersection of hypersurfaces $X_1,\ldots, X_r$ of degrees $d_1,\ldots,d_r$, and let $\Gamma^\prime, \Gamma^{\prime\prime}\subseteq \Gamma$ be closed subschemes residual to one another. Set $$m = -r - 1 + \sum_{d_i}.$$ Then for any $\ell\geq 0$, we have $$h^0(\PP^r, \mathcal{I}_{\Gamma^\prime}(\ell)) - h^0(\PP^r, \mathcal{I}_\Gamma(\ell)) = h^1(\PP^r, \mathcal{I}_{\Gamma^{\prime\prime}}(m - \ell)).$$
\end{theorem}

In principle, this equality could potentially be used as a way of detecting gaps in the sequences $\mathfrak{L}(v_d^{(n)}(\PP^n))$. However, as $n$ and $d$ grow, enumerating the possibilities for the dimensions of the involved cohomology groups appears to become difficult, and so using this result in that manner does not seem feasible.

We now show that the problem of computing ${\mathfrak{R}\mathfrak{L}}^G(v_d^{(n)}(\PP^n)),$ $\mathfrak{R}\mathfrak{L}(v_d^{(n)}(\PP^n))$ is equivalent to computing the maximal dimensions of the polynomial ring quotients above. First we show that any dimension obtained there can be realized as the number of points left in the intersection of hypersurfaces in $\PP^n$ that are each the union of $d$ hyperplanes.

We know that the ideal $I = (\prod_{i = 1}^d (x_0 - ix_n), \ldots, \prod_{i = 1}^d (x_{n-1} - ix_n))$ defines a complete intersection in $\PP^n$. The support of this intersection is the collection of $d^n$ distinct points $\{(a_0 :\ldots : a_{n-1} : 1)\mid a_i \in \{1,\ldots, d\}\}$, and so the complete intersection must be reduced. Adding any equation to the set of generators of the ideal $I$ gives an ideal which cuts out a reduced subscheme of $\PP^n$.

Next, observe there is a divisibility-preserving bijection $$\phi_n: \Omega_n \rightarrow \Omega_n^R$$ between $$\Omega_n^R := \left\{\prod_{i = 0}^{n-1} \prod_{j = 1}^{d_i}(x_i - j)\mid (d_0,\ldots,d_{n-1})\in \ZZ^n_{\geq 0}\right\}$$ and the set $\Omega_n$ of monomials in $k[x_0,\ldots,x_{n-1}]$, defined by $$\prod_{i = 0}^{n-1} x_i^{d_i}\mapsto \prod_{i = 0}^{n-1} \prod_{j = 1}^{d_i}(x_i - j).$$ That is, for $a, b\in \Omega_n$, $a$ divides $b$ if and only if $\phi_n(a)$ divides $\phi_n(b)$. With this notation, we have the following key fact.

\begin{lemma}
\label{reductionofmonideal}
Let $I = (x_0^d,\ldots, x_{n-1}^d, m_1,\ldots,m_r)$ be a monomial ideal such that $\deg(m_j)\leq d$ for each $j$. Then \begin{align*}\dim_k k[x_0,&\ldots,x_{n-1}]/I \\&= \dim_k k[x_0,\ldots,x_{n-1}]/(\phi_n(x_0^d),\ldots,\phi_n(x_{n-1}^d), \phi_n(m_1),\ldots,\phi_n(m_r)),\end{align*} as finite-dimensional vector spaces.
\end{lemma}
\begin{proof}
The dimension $$\dim_k k[x_0,\ldots,x_{n-1}]/(\phi_n(x_0^d),\ldots,\phi_n(x_{n-1}^d),\phi_n(m_1),\ldots,\phi_n(m_r))$$ equals the degree of the reduced scheme $V(\phi_n(x_0^d),\ldots,\phi_n(x_{n-1}^d), \phi_n(m_1),\ldots,\phi_n(m_r))$.

We know that the support of the reduced scheme $Y := V(\phi_n(x_0^d),\ldots,\phi_n(x_{n-1}^d))\subseteq \AS^n$ is the collection of $d^n$ distinct points $$\{(a_0,\ldots, a_{n-1})\mid a_i \in \{1,\ldots, d\}\}\subseteq \AS^n$$ as discussed above. The effect of intersecting $Y$ with a hypersurface of the form $V(\prod_{i = 0}^{n-1} \prod_{j = 1}^{d_i}(x_i - j))$ is to remove all points $(a_0,\ldots, a_{n-1})$ of $Y$ that satisfy $a_i > d_i$ for each $i$.

On the other hand, each point $(a_0,\ldots, a_{n-1})$  corresponds to a monomial $\prod_{i = 0}^{n-1}x_i^{a_i}$. The quotient ring $k[x_0,\ldots,x_{n-1}]/(x_0^d,\ldots,x_{n-1}^d)$ has as a basis the monomials corresponding to the points of $Y$. The effect of intersecting $Y$ with the hypersurface $V(\prod_{i = 0}^{n-1} \prod_{j = 1}^{d_i}(x_i - j))$ is analogous to forming the quotient ring $$k[x_0,\ldots,x_{n-1}]/(x_0^d,\ldots,x_{n-1}^d, \prod_{i = 0}^{n-1}x_i^{d_i});$$ the number of points left in $Y$ is equal to the dimension of this new quotient ring.

By repeating this process $r$ times, we obtain the desired result.
\end{proof}

Note the similarity between the patterns of Figures 1 and 2, after accounting for the difference in dimension. Making rigorous this similarity is all that Lemma \ref{reductionofmonideal} is serving to do. The quotient ring problem leads to a clean proof that $\mu(v_d^{(n)}(\PP^n)) = nd$.

\begin{lemma}
$\mu(v_d^{(n)}(\PP^n)) = nd$.
\end{lemma}
\begin{proof}
In the polynomial ring $k[x_0,\ldots,x_{n-1}]$ there are $N_d^{(n)}$ monomials of degree $\leq d$. 

The quotient ring $k[x_0,\ldots,x_{n-1}]/(x_0^d,\ldots, x_{n-1}^d)$ has dimension $d^n$ as a $k$-vector space, and has a basis that includes the elements represented by the monomials from $k[x_0,\ldots,x_{n-1}]$ of degree $\leq d$ other than the pure powers $x_0^d,\ldots, x_{n-1}^d$. It suffices to form a sequence $(a_j)_j$ of length $N_d^{(n)} - n - 1$ of those monomials so that $a_j$ does not divide $a_i$ for every $i>j$. This condition ensures that \begin{align*}\dim&_k(k[x_0,\ldots,x_{n-1}]/((x_0^d,\ldots, x_{n-1}^d) + (a_1,\ldots,a_j))) \\> &\dim_k(k[x_0,\ldots,x_{n-1}]/((x_0^d,\ldots, x_{n-1}^d) + (a_1,\ldots,a_{j+1})))\end{align*} for each $j$.

By Lemma \ref{reductionofmonideal}, this sequence corresponds to a sequence of hyperplanes $(H_1,\ldots,H_n)$ of $\PP^{N_d^{(n)} - 1}$ so that $$H_j \cap v_d^{(n)}(\PP^n) = V(\prod_{i = 1}^d(x_{j - 1} - ix_n))$$ for each $j = 1,\ldots, n$, and every other $H_j \cap v_d^{(n)}(\PP^n)$ for $j > n$ is formed as the union of irreducible components of the preceding hyperplane sections, possibly along with the extraneous component $V(x_n)$ which does not meet $H_1\cap\ldots\cap H_n\cap v_d^{(n)}(\PP^n)$. Thus the union of these hyperplane sections has exactly $nd$ irreducible components. Further, these hyperplanes all satisfy the requirements of $d$-reducibility.
\end{proof}

To compute ${\mathfrak{R}\mathfrak{L}}^G(v_d^{(n)}(\PP^n))$, $\mathfrak{R}\mathfrak{L}(v_d^{(n)}(\PP^n))$ we note that by symmetry, it suffices to consider only the sequences $(H_1,\ldots, H_{N_d^{(n)} - 1})\in\mathfrak{H}(X)$ with $$H_j\cap X = V(\prod_{i = 1}^d (x_{j-1} - ix_n))\subseteq \PP^n,$$ for $j = 1,\ldots, n$. For any $j\geq n$, we have that $$|X\cap (H_1\cap \ldots\cap H_j)| = \deg(X\cap (H_1\cap \ldots\cap H_j)) = \dim_k(k[x_0,\ldots,x_{n-1}]/I)$$ where $I$ is the homogeneous ideal of $X\cap (H_1\cap \ldots\cap H_j)$ dehomogenized with respect to $x_n$. Therefore, by Lemma \ref{reductionofmonideal}, we achieve our objective:

\begin{lemma} We have
\begin{enumerate}
\item \begin{align*}{\mathfrak{R}\mathfrak{L}}&^G(v_d^{(n)}(\PP^n)) \\&= \max\{(\dim_k(k[x_0,\ldots,x_{n-1}]/(x_0^d,\ldots, x_{n-1}^d, a_1,\ldots, a_{N_d^{(n)} - n - 1})), \\&\ldots, \dim_k (k[x_0,\ldots,x_{n-1}]/(x_0^d,\ldots, x_{n-1}^d)))\},\end{align*}
\item \begin{align*}{\mathfrak{R}\mathfrak{L}}&_{N_d^{(n)} - n - 1 - j}(v_d^{(n)}(\PP^n)) \\&= \max\{\dim_k(k[x_0,\ldots,x_{n-1}]/(x_0^d,\ldots, x_{n-1}^d, a_1,\ldots, a_j))\},\end{align*}
\end{enumerate}
where both maximums are taken over all monomials $a_k$ of degree $\leq d$ and so that $x_0^d,\ldots, x_{n-1}^d$, $a_1,\ldots, a_j$ are all linearly independent, and the first maximum is taken using the total order $\leq$ on integer sequences.
\end{lemma}

The question of whether these two sequences of secant indices are equal is then equivalent to asking whether for the integer sequences over which the maximum is taken in the expression for ${\mathfrak{R}\mathfrak{L}}^G(v_d^{(n)}(\PP^n))$ above, is a maximal element with respect to the total order $\leq$ also a maximal element with respect to the partial order $\preceq$? Our reduction of the problem to the quotient ring setup helps with the visualization of the combinatorics behind this question. In the $n = 2$ case, as in Example \ref{quarticexample}, the diagram is simple enough to see that the two sequences are the same. But that this is true in general is a fact which requires a careful proof.

We can think of the combinatorial problem of computing the maximal dimensions of these quotient rings as a specific instance of a more general family of problems. Consider a poset which is decomposed into the disjoint union of two finite sets $C = A\cup B$. Denote by $\leq$ the partial order on $C$ and let $m = |A|$. The problem is to consider all sequences of length $m - 1$ obtained by  picking an element of $A$, and then removing it and all elements greater than it according to $\leq$ from $C$, subject to the condition that the chosen element does not remove any further elements from $A$. That is, one considers all sequences $(a_1,\ldots, a_{m-1})$ of distinct elements of $A$ such that $a_j\not\leq a_{i}$ for all $i > j$, and then considers the sequence of cardinalities obtained by removing $a_1,\ldots,a_{m-1}$ one at a time, in that order, along with all elements larger than each.

In our computation of the secant indices, $C$ is the set of monomials that represent the generators of the quotient $k[x_0,\ldots,x_{n-1}]/(x_0^d,\ldots, x_{n-1}^d)$, and the set $A$ is that consisting of all monomials of degree $\leq d$ aside from the pure powers $x_0^d,\ldots, x_{n-1}^d$ in $k[x_0,\ldots,x_{n-1}]$.

It is simple to construct examples of such posets in general where the greedy sequence is not maximal with respect to $\preceq$.

\begin{example}
Consider the poset $\{1, a, b, b^2, c, a^2, a^2, a^3, a^4, a^5\}$, where $\alpha\leq \beta$ if and only if $\alpha$ divides $\beta$. Let $A = \{1, a, b, c, a^2, a^3\}$, and let $B = \{b^2, a^4, a^5\}$. Choosing the elements of $A$ in the order $c, b, a^3, a^2, a$ yields the ``greedy'' sequence $(9,8,6,3,2,1)$. Whereas removing the elements in the order $a^3, a^2, a, c, b$ yields the sequence $(9, 6, 5, 4, 3, 1)$, which surpasses the greedy sequence in the fourth term.
\end{example}

This poset seems uncomfortably close to the setup we work with to compute the reducible secant indices. However, no such discrepancy arises in our computations due to the following consequence \cite[Proposition 3.12]{mermin1} of the Clements-Lindstr\"om theorem.

\begin{lemma}
\label{consequencecl}
Let $R = \{f_1,\ldots,f_r\}\subseteq k[x_0,\ldots,x_{n-1}]$ be a regular sequence of monomials, with degrees $e_j = \deg(f_j), e_1\leq \ldots \leq e_r$. Let $N$ be any homogeneous ideal containing $R$. Then there exists a lex ideal $L\subseteq k[x_0,\ldots,x_{n-1}]$ such that $N$ and $(x_0^{e_1},\ldots, x_{r - 1}^{e_r}) + L$ have the same Hilbert series.
\end{lemma}

By a \emph{lex} ideal, we mean a monomial ideal $I$ of $k[x_0,\ldots,x_{n-1}]$ such that the degree $d$ piece of $I$, $I_d$, is generated by an initial segment of the lexicographic sequence of degree $d$ monomials, for every $d$. This solves the quotient ring dimension problem for us because of the following.

\begin{lemma}
Let $a_1,\ldots,a_r$, $1\leq r\leq N_d^{(n)} - n - 1$ be any sequence of monomials of degree $d$ in $k[x_0,\ldots,x_n]$ so that $x_0^d,\ldots,x_{n-1}^d,a_1,\ldots,a_r$ are all linearly independent. Let $b_1,\ldots, b_r$ be the initial segment of the lexicographic sequence of degree $d$ monomials in $k[x_0,\ldots,x_n]$, excluding the pure powers $x_0^d,\ldots,x_n^d$. Then letting $I = (x_0^d, \ldots, x_{n-1}^d) + (a^\prime_1,\ldots,a^\prime_r)$ and $J = (x_0^d, \ldots, x_{n-1}^d) + (b_1,\ldots,b_r)$, we have $$\dim_k(k[x_0,\ldots,x_{n-1}]/I^\prime)\leq \dim_k(k[x_0,\ldots,x_{n-1}]/J^\prime),$$ where $I^\prime, J^\prime$ denote $I,J$ dehomogenized with respect to $x_n$, respectively.
\end{lemma}
\begin{proof}
By Lemma \ref{consequencecl}, since $x_0^d,\ldots, x_{n-1}^d$ is a regular sequence, we know that there is a lex ideal $L$ such that $I$ has the same Hilbert series as the ideal $(x_0^d,\ldots,x_{n-1}^d) + L$. This Hilbert series is smaller at each term than that of the ideal $(x_0^d,\ldots,x_{n-1}^d) + L_d$. Since in particular the degree $d$ parts of the ideals $I$ and $(x_0^d,\ldots,x_{n-1}^d) + L$ must have the same dimensions as $k$-vector spaces we know the degree $d$ part of $(x_0^d,\ldots,x_{n-1}^d) + L_d$ is generated by $r + n$ distinct monomials.

Denoting by $b_1,\ldots,b_r$ these monomials other than the pure powers $x_0^d,\ldots,x_{n-1}^d$, in lexicographic order, we arrive at the desired result.
\end{proof}

This lemma proves that the two sequences of reducible secant indices and greedy indices are the same, and together with our previous observations in this section, shows their terms are identical to those that arise from the poset problem associated to the polynomial ring quotients above.

\begin{theorem}
Let $X = v_d^{(n)}(\PP^n)$, $n,d > 1$. Then $$\mathfrak{R}\mathfrak{L}(X) = {\mathfrak{R}\mathfrak{L}}^G(X),$$ and ${\mathfrak{R}\mathfrak{L}}_{N_d^{(n)} - 1 - n - j}(v_d^{(n)}(\PP^n)) = \dim_k(k[x_0,\ldots,x_{n-1}])/I^{(n,d)}_j$ where the ideal $I^{(n,d)}_j$ is the dehomogenization with respect to $x_n$ of the sum of the ideal $(x_0^d, \ldots, x_{n-1}^d)$ with the ideal generated by the lexicographic sequence of degree $d$ monomials excluding pure powers of length $N_d^{(n)} - n - j$ in $k[x_0,\ldots,x_n]$.
\end{theorem}

It would be interesting to have a simple formula that produces these sequences. For $n = 2$ or for $d = 2$, the corresponding sequences as the other number varies follow simple patterns. But as $n,d$ both grow these patterns become increasingly complex, as we will soon illustrate with several examples.

Our expectation is that these two sequences are in fact also equal to $\mathfrak{L}(X)$. We leave this as a conjecture.

\begin{conjecture}
\label{veroneseconjecture}
For each $n,d > 1$, $${\mathfrak{R}\mathfrak{L}}^G(X) = {\mathfrak{R}\mathfrak{L}}(X) = \mathfrak{L}(X),$$ where $X = v_d^{(n)}(\PP^n)$. In other words, the maximal number of points that can be contained in the intersection of $r$ linearly independent degree $d$ hypersurfaces in $\PP^n$ is ${\mathfrak{R}\mathfrak{L}}^G_{N_d - r - 1}(X)$, given that the intersection is finite and reduced.
\end{conjecture} 

If true, Conjecture \ref{veroneseconjecture} could be thought of a quasi-enumerative version of B\'ezout's theorem. There is some evidence for it from the currently open Eisenbud-Green-Harris conjecture, which has received a large amount of attention in the last few decades admitting only partial progress \cite{abedelfatah1} \cite{caviglia_maclagan1} \cite{eisenbud_green_harris1} \cite{gunturkun_hochster1}.

\begin{conjecture}[EGH \cite{eisenbud_green_harris1}]
Let $I$ be a homogeneous ideal in the polynomial ring $k[x_0,\ldots,x_{n-1}]$ containing a length $n$ regular sequence $f_1,\ldots, f_n$ of degrees $\deg(f_i) = a_i$, where $2\leq a_1\leq\ldots\leq a_n$. Then $I$ has the same Hilbert function as an ideal containing $x_0^{a_1},\ldots,x_{n-1}^{a_{n}}$.
\end{conjecture}

\begin{proposition}
The truth of the EGH-conjecture would imply that of Conjecture \ref{veroneseconjecture}.
\end{proposition}
\begin{proof}
Suppose we have any reduced complete intersection of degree $d$ hypersurfaces $T_1,\ldots,T_n$ in $\PP^n$, and let $T_{n+1},\ldots,T_{n + r}$ be degree $d$ hypersurfaces so that $T_1,\ldots,T_{n+r}$ are linearly independent. Write $T_j = V(F_j)$, $F_j\in k[x_0,\ldots,x_n]$ for each $j$. So the $F_1,\ldots,F_n$ form a regular sequence.

Suppose the EGH-conjecture is true. Then by \cite[Proposition 9]{caviglia_maclagan1} and Lemma \ref{consequencecl} see that the ideal $(F_1,\ldots,F_{n+r})$ has the same Hilbert series as an ideal of the form $(x_0^d,\ldots,x_{n-1}^d) + L$, where $L$ is a lex ideal. The degree $d$ part of $L$ must then be generated by the lexicographic sequence of monomials of degree $d$ excluding the pure powers and of length $r$. Thus $$\dim_k(k[x_0,\ldots,x_{n-1}]/I)\leq \dim_k(k[x_0,\ldots,x_{n-1}]/J),$$ where $I, J$ are the ideals $(F_1,\ldots,F_{n+r})$, $(x_0^d,\ldots,x_{n-1}^d) + L_d$ dehomogenized with respect to $x_n$.
\end{proof}

We conclude this section with several computations of the greedy sequence of secant indices and a specific instance of Conjecture \ref{veroneseconjecture}. For a given $n, d$, computing ${\mathfrak{R}\mathfrak{L}}^G(v_d^{(n)}(\PP^n))$ can be done easily using a computer algebra system such as SageMath \cite{sage1} by computing the degrees of the subschemes defined by the appropriate monomial ideals. In the following example, we use vertical bars to indicate where gaps of size $\geq 2$ occur in the sequences.

\begin{example}\label{exampleseqsveronese} Let $X = v_d^{(n)}(\PP^n)$.\\
$n = 2$

\begin{itemize}
\item $d = 2$, ${\mathfrak{R}\mathfrak{L}}^G(X) = (1, 2, 3, 4),$
\item $d = 3$, ${\mathfrak{R}\mathfrak{L}}^G(X) = (1, 2, 3, 4, 5, 6, 7 \mid9),$
\item $d = 4$, ${\mathfrak{R}\mathfrak{L}}^G(X) = (1, 2, 3, 4, 5, 6, 7, 8, 9, 10\mid 12, 13 \mid16),$
\item $d = 5$, \begin{align*}{\mathfrak{R}\mathfrak{L}}^G(X) =  &(1, 2, 3, 4, 5, 6, 7, 8, 9, 10, 11, 12, 13 \mid 15, \\&16, 17 \mid 20, 21 \mid 25),\end{align*}
\item $d = 6$, \begin{align*}{\mathfrak{R}\mathfrak{L}}^G(X) = &(1, 2, 3, 4, 5, 6, 7, 8, 9, 10, 11, 12, 13, 14, 15, \\&16 \mid 18, 19, 20, 21 \mid 24, 25, 26 \mid 30, 31 \mid 36),\end{align*}
\item $d = 7$, \begin{align*}{\mathfrak{R}\mathfrak{L}}^G(X) = &(1, 2, 3, 4, 5, 6, 7, 8, 9, 10, 11, 12, 13, 14, 15, \\&16, 17, 18, 19 \mid 21,22, 23, 24, 25 \mid 28, 29, \\&30, 31 \mid 35, 36, 37 \mid 42, 43 \mid 49).\end{align*}
\end{itemize}

$n = 3$

\begin{itemize}
\item $d = 2$, ${\mathfrak{R}\mathfrak{L}}^G(X) = (1, 2, 3, 4, 5, 6 \mid8),$
\item $d = 3$, \begin{align*}{\mathfrak{R}\mathfrak{L}}^G(X) = &(1, 2, 3, 4, 5, 6, 7, 9, 10, 11, 12, 13 \mid 15 \\ &\mid 18, 19 \mid 21 \mid 27),\end{align*}
\item $d = 4$, \begin{align*}{\mathfrak{R}\mathfrak{L}}^G(X) = &(1, 2, 3, 4, 5, 6, 7, 8, 9, 10 \mid 12, 13 \mid 16, \\ &17, 18, 19, 20, 21, 22\mid 24, 25 \mid 28 \mid 32, \\&33, 34 \mid 36, 37 \mid 40 \mid 48, 49 \mid 52 \mid 64),\end{align*}
\end{itemize}

$n = 4$

\begin{itemize}
\item $d = 2$, ${\mathfrak{R}\mathfrak{L}}^G(X) = (1, 2, 3, 4, 5, 6 \mid 8, 9, 10 \mid12 \mid 16),$
\item $d = 3$, \begin{align*}{\mathfrak{R}\mathfrak{L}}^G(X) = &(1, 2, 3, 4, 5, 6, 7\mid 9, 10, 11, 12, 13 \mid 15 \mid 18, 19 \\& \mid 21 \mid 27, 28, 29, 30, 31 \mid 33 \mid 36, 37 \mid 39 \mid 45 \\&\mid 54, 55 \mid 57 \mid 63 \mid 81),\end{align*}
\end{itemize}

$n = 5$

\begin{itemize}
\item $d = 2$, ${\mathfrak{R}\mathfrak{L}}^G(X) = (1, 2, 3, 4, 5, 6 \mid 8, 9, 10 \mid 12 \mid 16, 17, 18 \mid 20 \mid 24 \mid 32),$
\item $d = 3$, \begin{align*}{\mathfrak{R}\mathfrak{L}}^G(X) = &(1, 2, 3, 4, 5, 6, 7 \mid 9, 10, 11, 12, 13 \mid 15 \mid 18, 19 \mid 21 \mid 27, 28, \\& 29, 30, 31 \mid 33 \mid 36, 37 \mid 39 \mid 45 \mid 54, 55 \mid 57 \mid 63 \mid 81, 82, \\& 83, 84, 85 \mid 87 \mid 90, 91 \mid 93 \mid 99 \mid 108, 109 \mid 111 \mid 117 \mid 135 \\&\mid 162, 163 \mid 165 \mid 171 \mid 189 \mid 243).\end{align*}
\end{itemize}
\end{example}

Note when $d = 2$, the gaps in these sequences are similar to those observed in \cite[pg.~193]{eisenbud_green_harris1}. There the authors consider the maximal possible dimensions of quotients $k[x_0,\ldots,x_{n-1}]/I$ where $I$ is a homogeneous ideal generated by $r$ linearly independent quadrics. Because lower degree polynomials are not included as generators, their sequences are shorter. This difference becomes more pronounced if one considers the same numbers for ideals generated by linearly independent degree $d$ homogeneous polynomials rather than allowing lower degree generators as we do here.

When $n = 2$, the pattern determining ${\mathfrak{R}\mathfrak{L}}^G(v_d^{(n)}(\PP^n))$ is straightforward and yields a particularly attractive incarnation of Conjecture \ref{veroneseconjecture}. It is the continuation of the patterns observed in the above sequences for $n = 2$.

\begin{conjecture}
Fix $d>0$ and set $N = N^{(2)}_d$. Consider the sequence $(a_1,\ldots,a_{N - 3})$ with terms (organized first to last)
\begin{align*}
&d - 1\\
&1\\
&d - 2\\
&1, 1\\
&\ldots\\
& 4\\
&1,\ldots,1, \text{ (repeated } d - 4 \text{ times)}\\
&3\\
&1,\ldots,1, \text{ (repeated } d - 3 \text{ times)}\\
&2\\
&1,\ldots,1, \text{ (repeated } 3d - 3 \text{ times)}
\end{align*}

Let $(b_1,\ldots,b_{N - 2})$ be the sequence defined by $$b_j = d^n - \sum_{i = 1}^{j-1} a_{i}.$$ Then the maximal number of points that could be contained in the intersection of $r$ linearly independent degree $d$ curves is $b_{r - 1}$, given that the intersection is finite and reduced. 
\end{conjecture}

This is therefore also an implication of the EGH-conjecture, so any counterexample to it would also suffice to disprove the EGH-conjecture.


\section{Segre varieties}

In this section we will devise a method for computing the greedy reducible secant indices for the images of Segre embeddings. Let $n,m > 0$ and throughout denote by $\sigma_{n,m}: \PP^n\times\PP^m\hookrightarrow \PP^N$ the Segre embedding, where $N = (n + 1)(m+1) - 1$, and let $X = \sigma_{n,m}(\PP^n\times\PP^m)\subseteq \PP^N$. The degree of $X$ is $\binom{n + m}{n}$, so this is the top term of $\mathfrak{L}(X)$.

\begin{example}
When $m = 1$ is fixed and $n$ is allowed to vary (or vice versa), $X$ has dimension $n + 1$ inside $\PP^N$ where $N = 2(n + 1) - 1 = 2n + 1$. Thus $X$ has codimension $n$ and its degree is $\binom{n + 1}{n} = n + 1$. Therefore $X$ is a minimal degree subvariety of $\PP^N$ and so by Theorem \ref{genstrict} its sequence of secant indices is $$\mathfrak{L}(X) = (1,2,\ldots, n + 1).$$
\end{example}

Therefore the question of computing the sequences of secant indices only becomes interesting for $n,m > 1$. The first such example is $\PP^2\times\PP^2$ which is a codimension $4$ subvariety of $\PP^8$ of degree $6$. So its sequence of secant indices must contain exactly one gap of size $2$. In this section we will describe a method which can be used to compute ${\mathfrak{R}\mathfrak{L}}^G(X)$. These computations show that $${\mathfrak{R}\mathfrak{L}}^G(\sigma_{2,2}(\PP^2\times\PP^2)) = (1,2,3,4,6),$$ suggesting that the gap in the true sequence $\mathfrak{L}(\sigma_{2,2}(\PP^2\times\PP^2))$ occurs between the final two terms, like for the example of the degree $3$ Veronese embedding of $\PP^2$ into $\PP^9$. In fact, this example is simple enough that we can treat it by Theorem \ref{genstrict} and Proposition \ref{classificationextremal}, as $\sigma_{2,2}(\PP^2\times\PP^2)$ is smooth and nondegenerate, and is not a rational scroll.

\begin{proposition}
$$\mathfrak{L}(\sigma_{2,2}(\PP^2\times\PP^2)) = (1,2,3,4,6).$$
\end{proposition}

Proposition \ref{classificationextremal} could also be used to prove $\mathfrak{L}(v_3^{(2)}(\PP^2)) = (1,2,3,4,5,6,7,9)$ in place of the classical Cayley-Bacharach theorem. In this sense, this classification result, and results such as that of Noma \cite{noma1}, fulfill similar roles to the Cayley-Bacharach theorem but in the case of arbitrary smooth nondegenerate projective subvarieties. However, they only concern extremal secant subspaces and do not provide lower bounds for the secant indices, so do not give us a means to answer whether $$\mathfrak{L}(X) = {\mathfrak{R}\mathfrak{L}}^G(X)$$ for the more complicated instances of Segre and Veronese varieties.

Because of this, we focus here instead on a method of computing ${\mathfrak{R}\mathfrak{L}}^G(X)$ to obtain conjectural values for $\mathfrak{L}(X)$, for $X = \sigma_{n,m}(\PP^n\times\PP^m)$. To derive this method, note that if we choose coordinates, $\PP^n(x_0:\ldots : x_n)$, $\PP^m(y_0:\ldots :y_m)$, $\PP^N(z_0:\ldots : z_N)$ then for any hyperplane $H = V(a_0 z_0 + \ldots + a_N z_N)$ of $\PP^N$, we see that $X\cap H$ is the zero locus of the polynomial $a_0 x_0 y_0 + a_1 x_0 y_1 + \ldots + a_N x_n y_m$. This quadric can split into at most two factors.

On the other hand, we can choose any collection of $n+m$ hyperplanes $$H_1,\ldots, H_{n+m}\subseteq \PP^n$$ with the property that any subset of $n$ of these hyperplanes has only a single point in common, and any $n+1$ do not have any point in common. Likewise we can choose an analogous collection of hyperplanes $H^\prime_1,\ldots, H^\prime_{n+m}\subseteq \PP^m$ with the property that any subset of $m$ of those hyperplanes meet at a single point and no $m+1$ have a point in common. Each subscheme $$T_j := H_j\cup H^\prime_j\subseteq \PP^n\times\PP^m$$ is then a reducible and reduced hyperplane section of $\PP^n\times\PP^m$ via the Segre embedding, and we have $\bigcap_{j = 1}^{n+m}T_j$ is finite and reduced, realizing the degree $\binom{n+m}{n}$ of $\PP^n\times\PP^m$ in cardinality. Any $n+m$ hyperplanes of $\PP^N$ meeting at a dimension $N - n - m$ linear subvariety that meets $X$ in a finite and reduced collection of points will be of this form.

As in Section 4, there is symmetry here in the sense that the points of the intersection $T_1\cap \ldots \cap T_{n+m}$ are exactly those of the form $$(H_{a_1}\cap\ldots\cap H_{a_{n}})\times (H^\prime_{b_1}\cap \ldots\cap H^\prime_{b_m})$$ where $\{1,2,\ldots,n+m\} = \{a_1,\ldots,a_n\}\cup \{b_1,\ldots,b_m\}$. Therefore, we can strip the essence of the computation of ${\mathfrak{R}\mathfrak{L}}^G(X)$ from the context of the Segre embedding and find it equivalent to the following problem.

\begin{combprob}
\label{combprob2}
Let $Y$ be the set of all tuples $(A,B)$ where $A$ is a set of size $n$ and $B$ a set of size $m$ so that $A\cup B = \{1,2,\ldots,n+m\}$. Let $S$ be the set of all sequences of length $nm$ consisting of distinct tuples $(a,b)$ where $a,b\in \{1,2,\ldots,n + m\}$ and $a\neq b$. Given a tuple $(a,b)$, we say we are \emph{cutting $Y$ by $(a,b)$} if we replace $Y$ with the subset of elements $(A,B)$ of $Y$ for which either $a\in A$ or $b\in B$. For each sequence $((a_1,b_1),\ldots, (a_{nm}, b_{nm}))$, form an integer sequence $(c_0,\ldots,c_{nm})$ where $c_j$ is the cardinality of the set obtained by cutting $Y$ by each $(a_{i}, b_{i})$ for $i = 1, \ldots, nm - j$. Note $c_{nm} = |Y|$. The problem is then to compute the maximal (with respect to the total order $\leq$ of Section 3) possible integer sequence arising in this manner that is also strictly increasing.
\end{combprob}

The combinatorial problem gives us a clearer way to compute the reducibility of $X$ and find $\mu(X)$.

\begin{proposition}
Let $X = \sigma_{n,m}(\PP^n\times\PP^m)$. Then the reducibility of $X$ is $2$, and $\mu(X) = 2\dim(X)$.
\end{proposition}
\begin{proof}
Because of our observations above, we see that the most a hyperplane section of $X$ can split is into two components, so the reducibility of $X$ is at most $2$. All that must be done is exhibit a sequence of hyperplane sections satisfying the conditions of $2$-reducibility that also has the minimum total number of irreducible components.

To help simplify the notation, we use the notation of Combinatorial problem \ref{combprob2}. Let $p = (A,B)$ be any point of $Y$, and write $A = \{a_1,\ldots,a_n\}$, $B = \{b_1,\ldots,b_m\}$. Then consider the set $T$ consisting of the $nm$ points obtained from swapping one element of $A$ with one element of $B$. Use $p_{ij}$ to denote the point where $a_i$ was swapped with $b_j$. For each such $i,j$, cutting by the tuple $H_{i,j} = (a_i, b_j)$ removes $p_{ij}$ from $Y$ but does not remove any of the other elements of $T$.

Each $H_{i,j}$ corresponds to a hyperplane section of $X$, and the sequence of hyperplanes consisting of first the $n+m$ cutting out $Y$ and then of the $$H_{1,1}, H_{2,1},\ldots,H_{n,m},$$ taken in any order, is a sequence of hyperplane sections satisfying the conditions of $2$-reducibility, with $2\dim(X)$ total irreducible components.
\end{proof}

Thus altogether we have the following.

\begin{theorem}
For a given $n,m > 1$, the answer to Combinatorial problem \ref{combprob2} is the sequence ${\mathfrak{R}\mathfrak{L}}^G(X)$.
\end{theorem}


It is straightforward to write an algorithm that solves this problem for a given $n, m > 1$. We have implemented such an algorithm using the SageMath computer algebra system \cite{sage1}, and have used it to compute the greedy sequence of reducible secant indices for several values of $n,m$. It would be interesting to know if there is an analog of Lemma \ref{consequencecl} that would work in this context to show that $${\mathfrak{R}\mathfrak{L}}^G(X) = \mathfrak{R}\mathfrak{L}(X).$$ Verifying this without additional theoretical support requires a brute-force check of every possible sequence of elements of the form $(a,b)$ using the notation of Combinatorial problem \ref{combprob2}, and this becomes impractical even for small $n,m$. Similarly to Example \ref{exampleseqsveronese}, in the following example we use vertical bars to indicate gaps of size $\geq 2$ in the sequences.

\begin{example}
$m = n$
\begin{itemize}
\item $n = 2, m = 2$, ${\mathfrak{R}\mathfrak{L}}^G(X) = (1,2,3,4\mid6)$,
\item $n = 3, m = 3$, ${\mathfrak{R}\mathfrak{L}}^G(X) = (1, 2, 3, 4 \mid 6, 7 \mid 10, 11 \mid 14 \mid 20)$,
\item $n = 4, m = 4$, ${\mathfrak{R}\mathfrak{L}}^G(X) = (1, 2, 3, 4, 6, 7\mid 10, 11\mid 14 \mid 20, 21 \mid 25 \mid 35, 36 \mid 40 \mid 50 \mid70)$,
\end{itemize}

$m = n - 1$
\begin{itemize}
\item $n = 3, m = 2$, ${\mathfrak{R}\mathfrak{L}}^G(X) = (1, 2, 3, 4 \mid 6, 7 \mid10)$,
\item $n = 4, m = 3$, ${\mathfrak{R}\mathfrak{L}}^G(X) = (1, 2, 3, 4 \mid 6, 7 \mid10, 11 \mid14 \mid20, 21 \mid25 \mid35)$,
\end{itemize}

$m = n - 2$
\begin{itemize}
\item $n = 4, m = 2$, ${\mathfrak{R}\mathfrak{L}}^G(X) = (1, 2, 3, 4 \mid 6, 7 \mid 10, 11 \mid 15)$,
\item $n = 5, m = 3$, ${\mathfrak{R}\mathfrak{L}}^G(X) = (1, 2, 3, 4 \mid 6, 7 \mid 10, 11\mid 14 \mid 20, 21 \mid 25 \mid 35, 36 \mid 41 \mid 56)$,
\end{itemize}

$m = n - 3$
\begin{itemize}
\item $n = 5, m = 2$, ${\mathfrak{R}\mathfrak{L}}^G(X) = (1, 2, 3, 4 \mid 6, 7 \mid 10, 11 \mid 15, 16\mid 21)$.
\end{itemize}
\end{example}

It seems reasonable to expect that for this problem, like for that of the Veronese varieties, the greedy sequence is equal to the sequence of secant indices. We leave this as a conjecture.

\begin{conjecture}
For any $n,m > 1$, $X = \sigma_{n,m}(\PP^n\times\PP^m)$, we have $${\mathfrak{R}\mathfrak{L}}^G(X) = \mathfrak{L}(X).$$
\end{conjecture}

While it may be difficult to write down a closed form formula for ${\mathfrak{R}\mathfrak{L}}^G(X)$, one simple computation is determining the size of the last gap in ${\mathfrak{R}\mathfrak{L}}^G(X)$. A special case of the above conjecture is then that the same gap must be present in the sequence of secant indices.

\begin{conjecture}
$$\mathfrak{L}_{nm}(X) - \mathfrak{L}_{nm - 1}(X) = \binom{n + m - 2}{n - 1}.$$ That is, the maximum value of $|X \cap L|$ is $$\binom{n+m}{n} - \binom{n + m - 2}{n - 1}$$ for a linear subvariety $L$ of dimension $\codim(X) - 1 = nm - 1$ having finite, reduced intersection with $X$.
\end{conjecture} 

\section{Lines on surfaces}

Our original motivation comes from the classical problem of determining the maximal number $m_{\ell}(d)$ of lines that can be contained in a degree $d$ smooth surface $S$ of $\PP^3(x:y:z:w)$. As all surfaces of degree $d\leq 2$ are ruled, this question is only relevant for surfaces of degree $d\geq 3$. When $d = 3$, the Cayley-Salmon theorem ensures that every such $S$ must contain exactly $27$ distinct lines. However, for $d\geq 4$, the general degree $d$ surface contains no lines at all; the problem of determining $m_{\ell}(d)$ is a quasi-enumerative problem.

Over a century ago, Clebsch \cite{clebsch1} formulated the bound $m_{\ell}(d) \leq d(11d - 24)$ and Segre \cite{segre1} later proved $m_{\ell}(d) \leq (d - 2)(11d - 6)$, for all $d\geq 3$. These bounds have been improved slightly in modern times and the methods of Segre made rigorous using modern intersection theory; Bauer and Rams \cite[Theorem 1.1]{bauer_rams1} have proven $$m_{\ell}(d) \leq 11d^2 - 32d + 24$$ for all $d\geq 3$. To our knowledge, this is the best known bound for $m_{\ell}(d)$ when $d\geq 6$.

However, this latest bound is still known not to be sharp. It fails to be so already for the case of smooth quartic surfaces, where it is known that any irreducible quartic surface not ruled by lines contains at most $64$ lines, see for instance \cite[Theorem 4.5]{gonzalez_rams1}. Furthermore, quartics achieving this bound exist. To find a lower bound for $m_{\ell}(d)$ one only needs to provide an example of a surface with lines. For general $d$, the best known example of a surface with many lines is the smooth degree $d$ Fermat surface $V(x^d + y^d + z^d + w^d)$, which contains exactly $3d^2$ distinct lines.

The first of the numbers $m_{\ell}(d)$ which is currently unknown is that for $d = 5$. By \cite[Theorem 1.2]{rams_schutt1}, together with the example of the Fermat quintic, we see that $75\leq m_{\ell}(5) \leq 127$. To our knowledge, no example is currently known of a smooth quintic surface with more than $75$ lines. Interestingly, because $m_{\ell}(d)$ is bounded between two quadratic polynomials, if it turns out that $m_{\ell}(5) \leq 101$, then by interpolation using the examples $m_{\ell}(3) = 27, m_{\ell}(4) = 64$, there can be no polynomial function $f(d)$ that agrees with $m_{\ell}(d)$ for all $d\geq 3$. If this were true, then it would imply that any method which only produces a single polynomial bound is doomed to fail to produce an exact formula for $m_{\ell}(d)$.

Our work presented in this note began by looking for an alternate description of the numbers $m_{\ell}(d)$ along with methods that are capable of producing non-polynomial integer sequences. In what follows, we relax the constraint on the surfaces we consider to allow for nonreduced, reducible, and singular surfaces. Let $m_{\ell}(d)$ now denote the maximal possible number of lines that can be contained in a degree $d$ surface of $\PP^3$, given that the surface contains only finitely many lines.

There is a generalization of the Veronese embedding for Grassmannians \cite{harris1}; in particular, given a $d > 1$, one may define an embedding $$v_d: \GG(1,3) = G(2,4)\hookrightarrow G\left(\binom{d + 3}{d} - d - 1, \binom{d + 3}{d}\right)$$ by $$L\mapsto I(L)_d,$$ for each line $L\subseteq \PP^3$, where $I(L)_d$ denotes the degree $d$ part of the homogeneous ideal $I(L)$ of $L$.

Given a nonzero homogeneous polynomial $F\in k[x,y,z,w]$ of degree $d$, one can consider the subset of $G(\binom{d + 3}{d} - d - 1, \binom{d + 3}{d})$ consisting of all $\binom{d + 3}{d} - d - 1$-planes of $k[x,y,z,w]_d$ containing $F$. This is a special type of Schubert subvariety of $G(\binom{d + 3}{d} - d - 1, \binom{d + 3}{d})$, a sub-Grassmannian, isomorphic to $G(\binom{d + 3}{d} - d - 2, \binom{d + 3}{d} - 1)$. Then note that we can express $m_{\ell}(d)$ as the maximal finite intersection that can occur by intersecting $v_d(\GG(1,3))$ with such sub-Grassmannians, $$m_{\ell}(d) = \max\{|v_d(\GG(1,3))\cap G|\mid |v_d(\GG(1,3))\cap G|<\infty\},$$ where this maximum is taken over all sub-Grassmannians $G$ in $G(\binom{d + 3}{d} - d - 1, \binom{d + 3}{d})$ of the above form.

All this generalized Veronese embedding serves to do is provide us with an alternate language with which to state our problem. However, viewing the problem from this perspective suggests a way to relate $m_{\ell}(d)$ to several integer sequences that seem interesting in their own right. One sequence works as follows.

\begin{definition}
For each $i$-plane $V$ in $\PP^n$, denote by $G_{V}\cong G(k - i, n - i)$ the sub-Grassmannian of $\GG(k, n) = G(k + 1, n + 1)$ consisting of all $k$-planes of $\PP^n$ containing $V$. Let $X\subseteq \GG(k,n)$ be any subvariety of dimension $\leq \codim(\GG(k - 1, n - 1))$. Supposing that $X$ has finite intersection with a $G_V$ for at least one point $V$ in $\PP^n$, for each $i = 0,\ldots, k - 1$, we may define $$\mathfrak{L}_i(X) := \max\{|X\cap G_V|\mid |X\cap V| < \infty\},$$ where this maximum is taken over all $i$-planes $V$ of $\PP^n$.
\end{definition} 

The reason for the abuse of notation in redefining $\mathfrak{L}(X)$ here is that when $k$ = $n - 1$, so $\GG(k,n)$ is a projective space, this definition indeed specializes to our first definition of the secant indices described in Section 2, just without the condition that $X$ be smooth or that these intersections be reduced.

In the specific case of interest, when $X = v_d(\GG(1,3))$ and $\GG(k,n)$ is the codomain of $v_d$, we have that the bottom term of the resulting integer sequence (when $i = 0$) is exactly $m_{\ell}(d)$. As $G_V$ is a single point when $V$ is $k$-dimensional, the top term of the sequence is just $1$. Described differently, the question of what integers appear in this sequence is equivalent to the following.

\begin{question}
What is the maximal number of lines that can be contained in the intersection of $1 \leq i \leq \binom{d + 3}{d} - d - 1$ linearly independent degree $d$ surfaces (possibly singular, reducible, or nonreduced) of $\PP^3$, supposing that the intersection contains only finitely many lines?
\end{question}

The fact that the intersection of any $\binom{d + 3}{d} - d - 1$ degree $d$ independent surfaces can only contain $1$ line is a consequence of the fact that $\dim_k(I(L)_d) = \binom{d + 3}{d} - d - 1$ for any line $L\subseteq \PP^3$. One could further modify the question to only allow the intersection of smooth degree $d$ surfaces. In this case, it is straightforward to see that the terms of the resulting integer sequence are nondecreasing as one varies $i$ from $\binom{d + 3}{d} - d - 1$ to $1$. The last term of this sequence is the $m_{\ell}(d)$ defined at the beginning of this section for smooth surfaces, and the penultimate term is at most $d^2$ by B\'ezout's theorem.

As an example, for $d = 4$, this sequence has $\binom{d + 3}{d} - d - 1 = 30$ terms, but we already know the penultimate term is at most $16$ and the first term is $1$. So the sequence must contain repeated numbers. Our original hope was that if one could compute the earlier numbers of these sequences, then it would be possible to extrapolate from observable patterns in those numbers a formula for $m_{\ell}(d)$. The feasibility of this approach appears dubious, but these sequences seem to be of independent interest, potentially reflecting properties of the possible configurations of lines on degree $d$ surfaces.

\end{document}